\documentclass[preprint, 12pt]{elsarticle}
  \usepackage{amsmath}
  \usepackage{amssymb}
  \usepackage{bbm}
  \usepackage{graphicx}
	\usepackage{wasysym}
\usepackage{subfig}
\newdefinition{remark}{Remark}
\newdefinition{proposition}{Proposition}
\newdefinition{proof}{Proof}
\newdefinition{algorithm}{Algorithm}

\newcommand{\du}{\, \mathrm{d}}
\newcommand{\myremarkend}{~\hfill$\spadesuit\/$}

\journal{...}
\begin{document}

\begin{frontmatter}

\title{A Sharp-Interface Active Penalty Method for the Incompressible Navier-Stokes Equations.}

\author{D. Shirokoff}
\ead{david.shirokoff@mail.mcgill.ca}

\author{J.-C. Nave}
\ead{jcnave@math.mcgill.ca}

\address{Department of Mathematics and Statistics, McGill University,
         Montreal, QC H3A0B9, CAN}

\begin{abstract}
	The volume penalty method provides a simple, efficient approach for solving the incompressible Navier-Stokes equations in domains with boundaries or in the presence of moving objects.  Despite the simplicity, the method is typically limited to first order spatial accuracy.  We demonstrate that one may achieve high order accuracy by introducing an \emph{active penalty term}.  One key difference from other works is that we use a sharp, unregularized mask function.  We discuss how to construct the active penalty term, and provide numerical examples, in dimensions one and two.  We demonstrate second and third order convergence for the heat equation, and second order convergence for the Navier-Stokes equations. In addition, we show that modifying the penalty term does not significantly alter the time step restriction from that of the conventional penalty method.
\end{abstract}


\end{frontmatter}

\section{Introduction}

There are many popular methods for numerically solving the incompressible Navier-Stokes equations in complex geometries.  For instance, the immersed boundary method \cite{Peskin2000}, the immersed interface method \cite{LeKhooPeraire2006} and the ghost fluid method  \cite{FedkivAslamMerrimanOsher1999} are popular since they allow one to use a regular grid with an immersed domain boundary.  Other efficient methods for the Navier-Stokes or heat equation include \cite{GibouChenNguyenBanerjee2007,GibouFedkiw2005,NgMinGibou2009}.  These methods not only use a regular grid, but also utilize level set functions to ensure a sharp interface.  In all cases, the regular grid and level set formulation alleviates many of the numerical difficulties introduced by curved or moving boundaries.  In this paper, we focus on the volume penalty method \cite{Angot2005,Angot2010,Bruneau2000,BruneauFabrie1996,KhadraAngotParneixCaltagirone2000}, which loosely fits into the same class of methods.

As a result of their simplicity, penalty methods have been used in a wide variety of problems including electromagnetism, magnetohydrodynamics \cite{MoralesLeroyBosSchneider2012}, shape optimization \cite{ChantalatBruneauGalusinskiIollo2009}, fluid-solid interaction problems \cite{CoquerelleCottet2008,KadochKolomenskiyAngotSchneider2012} and even simpler problems such as the heat equation or Poisson equation \cite{RamiereAngotBelliard2007}.  In the context of fluids, they provide a simple means for solving the incompressible Navier-Stokes equations in domains with boundaries.  The approach relies on replacing the often difficult to implement Dirichlet fluid boundary conditions, with a simpler to implement volumetric forcing term in the advection equation.  

Despite the simplicity, the penalty method suffers from i) poor convergence in the penalty parameter, thereby restricting the accuracy of numerical methods and, ii) a lack of regularity in the velocity field which reduces the advantages of spectral methods.  For example, solutions to the penalized equations have a discontinuous second derivative which limits the decay rate of the Fourier coefficients, as well as the ability to spectrally compute derivatives.  Despite the lack of smoothness, stable and low order spectral methods have been successfully used to solve the penalized fluid equations \cite{KadochKolomenskiyAngotSchneider2012,KolomenskiySchneider2009}. 

The focus of our paper is to introduce a systematic method for improving the accuracy of penalty methods.  Current methods which improve accuracy rely on introducing a subgrid numerical construct in the vicinity of the domain boundary \cite{SarthouVincentCaltagirone2011,SarthouVincentCaltagironeAngot2008}.  Such approaches, however, are restrictive if one wishes to eventually use high order Fourier methods.  One distinct difference with our approach is that we alter the equations at the continuous level to improve the analytic convergence rate of the penalized problem to the original unpenalized problem.  The improved analytic convergence rate then allows for higher order numerical schemes.

We first introduce the original volume penalty method, followed by an introduction to the improved \emph{active penalty method}.  We then explicitly show how to analytically construct the new penalty term.  Following the construction, we then examine a model equation to demonstrate how the active penalization improves the convergence rate for the Poisson equation.  

After discussing the improved convergence, we focus on numerical details.  First, we examine the stability of the new active penalty term, and show that it does not introduce additional numerical stiffness. We then provide numerical examples for the heat equation, in dimensions one and two, showing second and third order schemes. Lastly, we outline how to handle the divergence constraint for the Navier-Stokes equations and provide numerical examples showing second order convergence (in $L^{\infty}$) in the velocity field and first order in the pressure.

\section{Navier-Stokes and volume penalty equations} \label{Section_NavierStokes}

The aim of our work is to examine the behavior of a fluid in the vicinity of a solid or a porous medium.  For instance, two examples include the motion of a fluid in a bounded domain with hard walls, or the motion around an immersed solid body such as the one shown in figure \ref{Obstacle}.  In our case, we consider dimensions $D = 2, 3$ and let $\Omega_p \subset \mathbbm{R}^D$ denote the physical fluid domain.  For our purposes, $\Omega_p$ is an open set with $C^2$ boundary $\Gamma = \partial \Omega_p$.  

\subsection{Incompressible Navier-Stokes equations}
Through the conservation of mass and momentum, the incompressible Navier-Stokes equations govern the flow of an incompressible fluid for $\mathbf{x} \in \Omega_p$
\begin{eqnarray} \label{NS_Eq1}
	\partial_t \mathbf{u} + \mathbf{u}\cdot\nabla \mathbf{u} &=& -\nabla p + \mu \Delta \mathbf{u} + \mathbf{f} \\ \label{NS_Eq2}
	\nabla \cdot \mathbf{u} &=& 0. 
\end{eqnarray}
Here $\mathbf{u}(\mathbf{x}, t)$ is the velocity vector field, $p(\mathbf{x}, t)$ is the pressure, $\mu > 0$ is the kinetic viscosity, and $\mathbf{f}(\mathbf{x}, t)$ is an external forcing such as gravity.  

To supplement the bulk equations (\ref{NS_Eq1})--(\ref{NS_Eq2}), the fluid velocity also satisfies prescribed boundary conditions
\begin{eqnarray} \label{NS_BC}
	\mathbf{u} &=& \mathbf{g} \hspace{10mm} \mbox{for } \mathbf{x} \in \Gamma \\ \label{NS_BC_Constraint}
	\int_{\Gamma} \mathbf{g}\cdot \mathbf{n} \du A &=& 0
\end{eqnarray}
Here $\mathbf{n}$ is an outward pointing normal, while equation (\ref{NS_BC_Constraint}) represents a consistency condition on the boundary data.  Although we allow $\mathbf{g} = \mathbf{g}(\mathbf{x}, t)$ to be a function of both space and time, the case of $\mathbf{g} = 0$ represents the practical condition of a no-slip and no-flux boundary condition.  Together, equations (\ref{NS_Eq1})--(\ref{NS_Eq2}) with boundary data (\ref{NS_BC}) describe the evolution of an initial, divergence-free velocity field $\mathbf{u}(\mathbf{x}, 0) = \mathbf{u}_0(\mathbf{x})$.  

\subsection{Volume penalty equations}
Domains with curved boundaries $\Gamma$ present several challenges to the numerical solution of equations (\ref{NS_Eq1})--(\ref{NS_Eq2}).  For example, curved boundaries or immersed objects limit the use of Fourier methods since solutions are not periodic, or easily extended to periodic functions.  One simple solution to handle complicated or moving boundaries is through the use of a volume penalty term in the Navier-Stokes equations.  In such a case, one removes the Navier-Stokes boundary condition, and instead adds a drag term to the momentum equation.  

To introduce the penalized equation, we first denote $\Omega_s$ as the solid domain of the obstacle or wall.  Here the obstacle region $\Omega_s = \overline{\Omega}_s$ is a closed set which shares the same boundary as the fluid,  $\partial \Omega_s = \Gamma$.  The penalized equations are then defined on a computational domain $\Omega$ which is the union of the physical and solid domains $\Omega = \Omega_p \bigcup \Omega_s$.  In our case we take $\Omega$ to be a rectangular domain with periodic boundary conditions, i.e. $\Omega = \mathbbm{T}^D$ where $\mathbbm{T}^D$ is the D-dimensional torus.  

\begin{figure}[htb!]
	\centering
    \includegraphics[width = 0.9\textwidth]{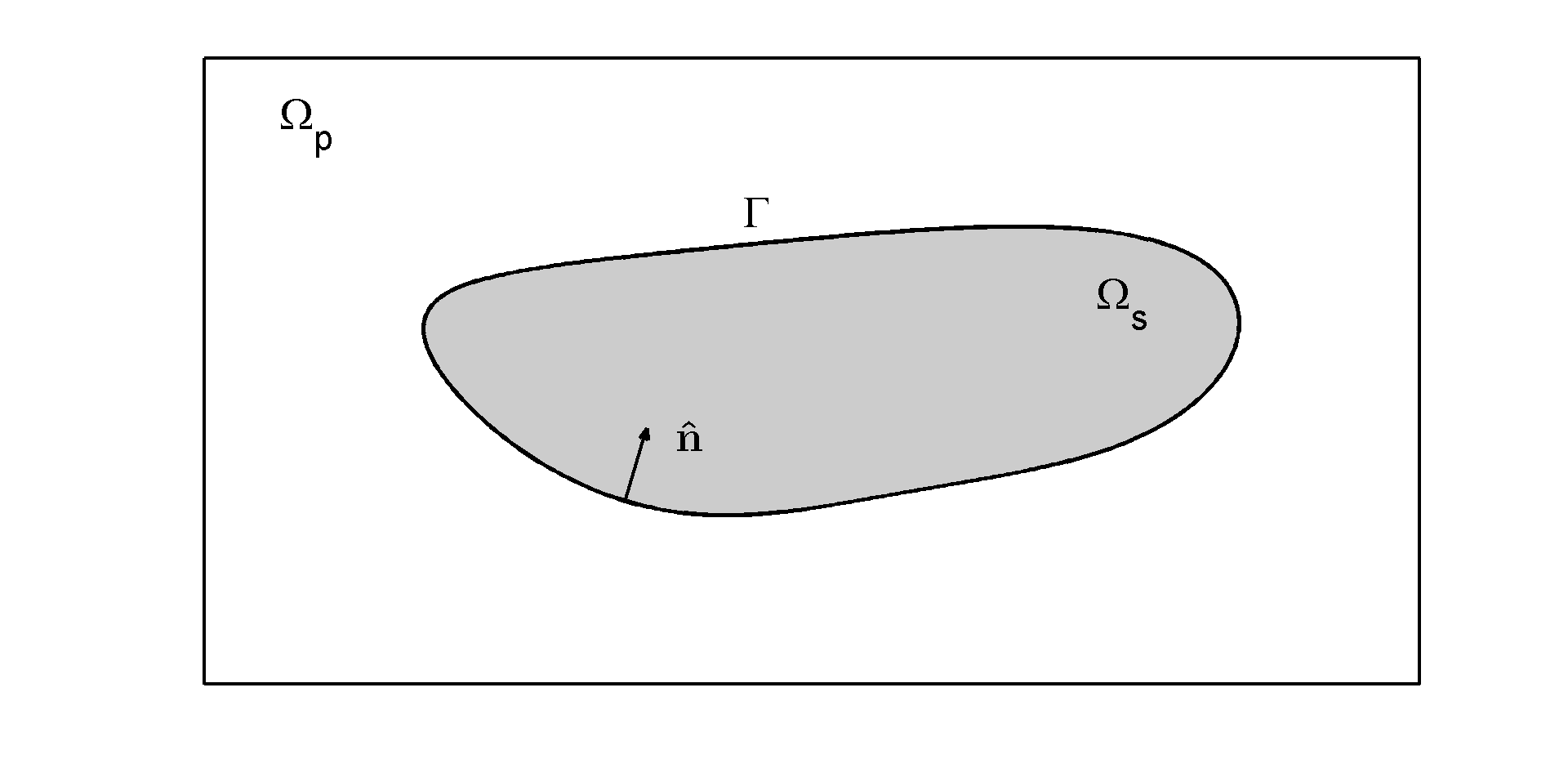} \\
    \caption{Illustration of the physical fluid ($\Omega_p$) and solid obstacle ($\Omega_s$) domains.}
 \label{Obstacle}
\end{figure}

For a stationary obstacle with a $\mathbf{g} = 0$ boundary condition, the volume penalty equations (see \cite{AngotBruneauFabrie1999,AngotCaltagiron1990,ArquisCaltagirone1984}) are
\begin{eqnarray} \label{Penalty_Eq1}
	\partial_t \mathbf{u}_{\eta} + \mathbf{u}_{\eta}\cdot\nabla \mathbf{u}_{\eta} &=& -\nabla p_{\eta} + \mu \Delta \mathbf{u}_{\eta} + \mathbf{f} - \eta^{-1}\chi_s \; \mathbf{u}_{\eta} \hspace{10mm} \mathbf{x} \in \Omega \\ \label{Penalty_Eq2}
	\nabla \cdot \mathbf{u}_{\eta} &=& 0. 
\end{eqnarray}
Here $\eta$ is a small parameter, and $\chi_s(\mathbf{x})$ is the characteristic function on $\Omega_s$, namely 
\begin{eqnarray} \label{CharFunction}
	\chi_s(\mathbf{x}) = \begin{cases}
		0 & \text{for } \mathbf{x} \in \Omega \setminus \Omega_s \\
		1 & \text{for } \mathbf{x} \in \Omega_s
	\end{cases}.
\end{eqnarray}
In the limit $\eta \rightarrow 0$, the drag term in equation (\ref{Penalty_Eq1}) becomes large and tends to slow the fluid inside $\Omega_s$.  Rigorous convergence results by Angot et al. \cite{AngotBruneauFabrie1999}, and Carbou and Fabrie \cite{CarbouFabrie2003} show that the penalized velocity $\mathbf{u}_{\eta}$ converges to the solution of the Navier-Stokes equations $\mathbf{u}$ with an error rate of $O(\eta^{1/2})$ in the $L^2(\Omega_p)$ norm.

\subsection{Improved volume penalty equations}

Although the volume penalty equations do converge to Navier-Stokes as $\eta \rightarrow 0$, the convergence rate is slow and therefore may limit the accuracy of resulting numerical schemes.  For example, let $\mathbf{u}_{\eta, num}$ denote a numerical solution for the penalized equations. One is then interested in quantifying the numerical error for $\mathbf{u}_{\eta, num}$ compared to $\mathbf{u}$, the solution to the original Navier-Stokes problem (\ref{NS_Eq1})--(\ref{NS_Eq2}).  Using the triangle inequality\footnote{Here $||\cdot||_2$ is any appropriate numerical $L^2(\Omega_p)$ norm.}, the error can be controlled by
\begin{eqnarray}
	||\mathbf{u} - \mathbf{u}_{\eta, num}||_2 \leq ||\mathbf{u} - \mathbf{u}_{\eta}||_2 + ||\mathbf{u}_{\eta} - \mathbf{u}_{\eta, num}||_2.
\end{eqnarray}
Rigorous convergence results then bound the first term as 
$||\mathbf{u} - \mathbf{u}_{\eta}||_{2} \sim \eta^{1/2}$, while $||\mathbf{u}_{\eta} - \mathbf{u}_{\eta, num}||_2$ depends on the numerical details and order of the scheme.  Finally, we note that the addition of the penalty term introduces time scales of $O(\eta)$ and length scales of $O(\eta^{1/2})$ into the solution $\mathbf{u}_{\eta}$.  To appropriately resolve the boundary layers in the penalty equations  (\ref{Penalty_Eq1})--(\ref{Penalty_Eq2}), one then has a grid spacing of $\Delta x \sim \sqrt{\eta}$ leading to a first order bound
\begin{eqnarray}
	||\mathbf{u} - \mathbf{u}_{\eta, num}||_2 \leq O(\Delta x).
\end{eqnarray}
In light of the above observations, a high order penalty method must either increase the boundary layer width $O(\sqrt{\eta})$, or improve the analytic convergence rate in the penalty parameter.  We adopt the second approach, and outline how equation (\ref{Penalty_Eq1}) can be modified to better approximate the original Navier-Stokes problem (\ref{NS_Eq1})--(\ref{NS_Eq2}).  Furthermore, we note that when modifying the penalty term, it is important to avoid the introduction of additional length or time scales which would hinder the development of high order numerical schemes.  To improve the penalized equations, we exploit the fact that $\mathbf{u}$ satisfies the boundary conditions on $\Gamma$, and does not represent a physical flow inside $\Omega_s$.  Specifically, we modify the penalty term so that the flow tracks an extension function $\mathbf{\tilde{g}}$ defined on $\mathbf{x} \in \Omega_s$.  In such a case, the volume penalty equations take the form 
\begin{eqnarray} \label{ImprovedPenalty_Eq1}
	\partial_t \mathbf{u}_{\eta} + \mathbf{u}_{\eta}\cdot\nabla \mathbf{u}_{\eta} &=& -\nabla p_{\eta} + \mu \Delta \mathbf{u}_{\eta} + \mathbf{f} - \eta^{-1}\chi_s \; (\mathbf{u}_{\eta} - \mathbf{\tilde{g}})\hspace{4mm} \mathbf{x} \in \Omega \\ \label{ImprovedPenalty_Eq2}
	\nabla \cdot \mathbf{u}_{\eta} &=& 0 \hspace{6mm} \mathbf{x} \in \Omega_p. 
\end{eqnarray}

At this point, we only specify the divergence constraint within the physical domain and defer a more detailed description of the divergence constraint inside $\Omega_s$ for section \ref{Section_NumericsNSEq}.  The idea is to choose $\mathbf{\tilde{g}}$ to reduce the artificial fluid boundary layer generated by the penalized equations in the vicinity of $\Gamma$.  Specifically, the function $\mathbf{\tilde{g}}$ should be a smooth, at least $C^1$, extension of the prescribed boundary conditions.  The extension is constructed for each component of $\mathbf{\tilde{g}}$, and each component of $\mathbf{\tilde{g}}$ should be chosen to match $k$ derivatives of $\mathbf{u}$ in the direction normal to $\Gamma$. As a result, we prescribe the following general properties for $\mathbf{\tilde{g}}$ 
\begin{enumerate}
	\item [P1.] $\mathbf{\tilde{g}}$ is an extension of the prescribed boundary values: $\mathbf{\tilde{g}} = \mathbf{g}$ for $\mathbf{x} \in \Gamma$.
	\item [P2.] $\mathbf{\tilde{g}}$ has the same normal slope as $\mathbf{u}$: $(\mathbf{n}\cdot \nabla) u_i = (\mathbf{n} \cdot \nabla) \tilde{g}_i$. Here $u_i$ and $\tilde{g}_i$ for $i = 1\ldots D$ are the components of $\mathbf{u}$ and $\mathbf{\tilde{g}}$
	\item [P3.] For higher derivatives, $(\mathbf{n}\cdot \nabla)^k u_i |_{\Gamma_p} =  (\mathbf{n}\cdot \nabla)^k \tilde{g}_i$.  
\end{enumerate}
Since derivatives of $\mathbf{u}$ may be discontinuous across $\Gamma$, the notation $\Gamma_p$ denotes the limit of the derivative from the physical domain $\Omega_p$.

\section{Constructing the extension $\mathbf{\tilde{g}}$} \label{Section_Extension}
In this section we discuss one possible construction for the extension function $\mathbf{\tilde{g}}$.  The construction procedure is identical for each component $g_i$ of $\mathbf{\tilde{g}}$ for $i = 1\ldots D$. 
 
In our approach, we assume the domain $\Omega_p$ has a smooth boundary, at least $\Gamma \in C^2$. As a result, we omit a class of physically important domains such as rectangles.  The general idea is to match the normal derivatives of $\mathbf{\tilde{g}}$ to those of $\mathbf{u}$ on $\Gamma$.  With the appropriate boundary derivatives, we then let $\mathbf{\tilde{g}}$ decay to some constant value $\mathbf{G}$ over a length scale $l$.  In our construction, the maximum length scale $l$ is bounded by the minimum radius of curvature of the interface.  

\begin{enumerate}
\item [Step 1.]
First introduce a family of smooth, one-dimensional basis functions $B_j \in C^{k}$ with $0 \leq j \leq k$ such that
\begin{enumerate}
	\item [(i)] The functions $B_j$ form a basis for derivatives at $x = 0$
\begin{eqnarray} \label{KDelta}
	\frac{d^i}{dx^i} B_j(0) = \begin{cases}
		1 & \text{for } j = i \\
		0 & \text{for } j \neq i
	\end{cases}
\end{eqnarray}
	\item [(ii)] Each $B_j(x)$ has support on $[0, 1]$. Namely $B_j(x) = 0$ for $x < 0$ and $x > 1$.
\end{enumerate}
One can then use the functions $B_j(x)$ to construct a $C^k$ extension $\tilde{f}(x)$ of a one-dimensional function $f(x)$ on $x \geq 0$ as
\begin{eqnarray} \label{OneDExtension}
	\tilde{f}(x) = f(0) B_0(x) + f'(0) B_1(x) + \ldots + f^{(k)}(0) B_k(x).
\end{eqnarray}
Note that by construction, the function $\tilde{f}(x)$ matches k derivatives at $x = 0$ and vanishes for $x > 1$. The goal is now to modify the extension (\ref{OneDExtension}) to higher dimensions.

Although there are many different choices for $B_j(x)$, we now give an example of one such choice for matching $k = 2$ derivatives.  We do this by constructing $B_j(x)$ out of stretched copies of the smoothly decaying function
\begin{eqnarray}
	h(x) = \begin{cases}
		e^{1 - \frac{1}{1-x}} & \text{for } 0 \leq x < 1 \\
		0 & \text{for } x \geq 1
	\end{cases}.
\end{eqnarray}
Using $h(x)$, one can take the functions $B_0, B_1, B_2$ (figure \ref{Basisfunctions}) as the weighted sums
\begin{eqnarray}
	B_0(x) &=& 3 \; h(x) - 3 \;h(2x) + h(3x) \\
	B_1(x) &=& \frac{5}{2}\; h(x) - 4\; h(2x) + \frac{3}{2}\; h(3x) \\
	B_2(x) &=& -\frac{1}{2}\; h(x) + h(2x) - \frac{1}{2}\; h(3x).
\end{eqnarray}

\begin{figure}[htb!]
	\centering
    \includegraphics[width = \textwidth]{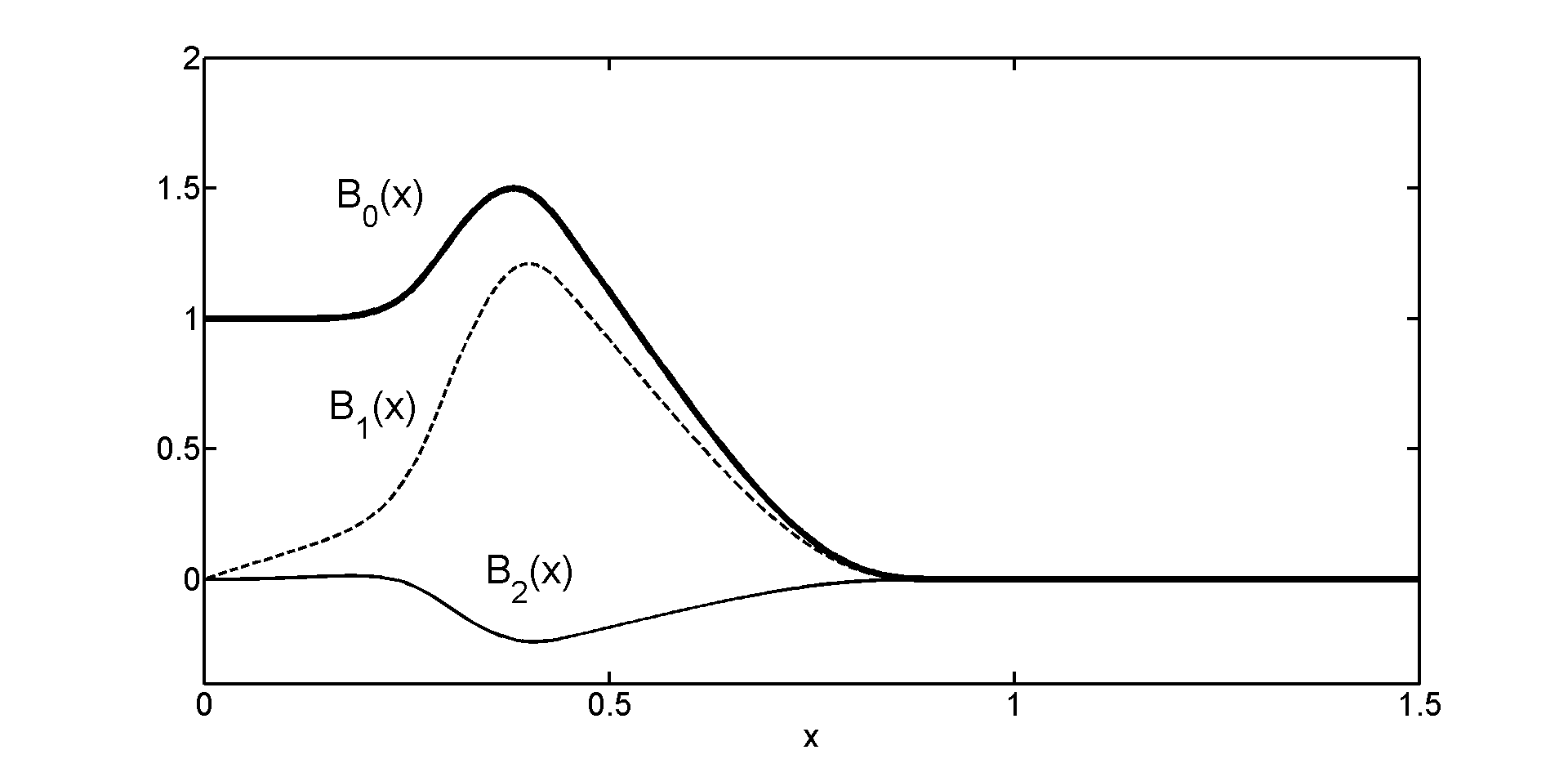} \\
    \caption{A plot of the $1D$ basis functions $B_0(x)$ (thick), $B_1(x)$ (dashed) and $B_2(x)$ (thin).}
 \label{Basisfunctions}
\end{figure}

\item [Step 2.] 
Construct a coordinate system inside the obstacle. The coordinate system should be orthogonal at the boundary, and only needs to extend a distance $l$ inside the domain $\Omega_s$.  

To construct the coordinates, we follow a standard approach \cite{Folland1995} shown in figure \ref{LocalCoordinates}.  Let $\boldsymbol \xi \in \Gamma$ denote the coordinates of the boundary $\Gamma$.  Any point $\mathbf{x}$ within a distance $l$ of the boundary may then be written as
\begin{eqnarray} \label{ForwardCoordinates}
	\mathbf{x} = \boldsymbol \xi + s \, \mathbf{n}(\boldsymbol \xi).
\end{eqnarray}
Here $s$ is the normal distance inside $\Omega_s$ from the boundary.  Within a small enough region, $s \leq l$, one may invert\footnote{The coordinates $\boldsymbol \xi(\mathbf{x})$ and $s(\mathbf{x})$ are both at least $C^1$ functions} equation (\ref{ForwardCoordinates}) to write $\boldsymbol \xi = \boldsymbol\xi(\mathbf{x})$ and $s = s(\mathbf{x})$.

\begin{figure}[htb!]
	\centering
    \includegraphics[width = 0.7\textwidth]{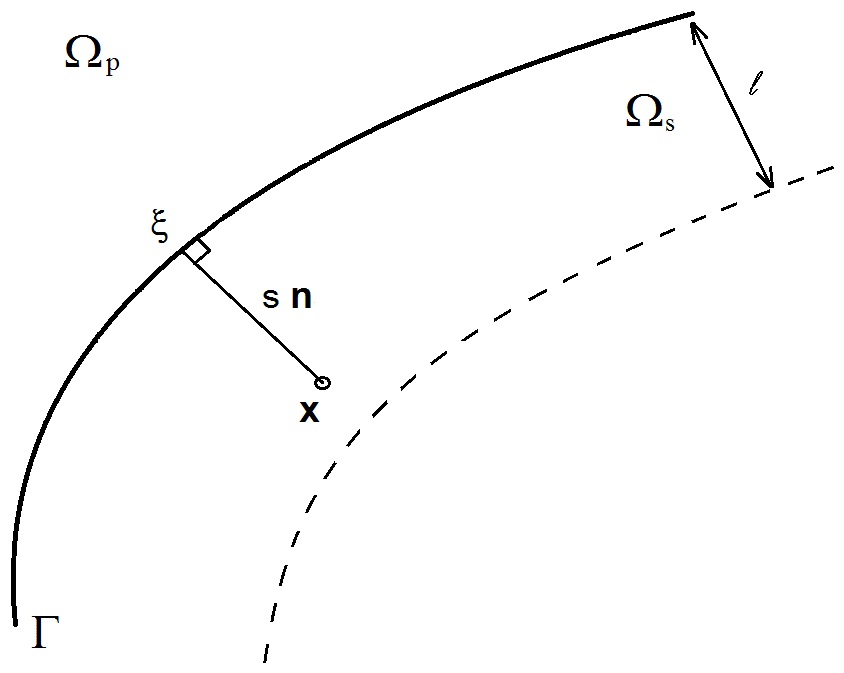} \\
    \caption{A local set of normal coordinates. Here $\boldsymbol\xi \in \Gamma$ is a point on the boundary, while $s$ is the distance in the normal direction. A coordinate inside a neighborhood of $\Omega_s$ has the form $\mathbf{x} = \boldsymbol \xi + s\, \mathbf{n}(\boldsymbol\xi)$.}
 \label{LocalCoordinates}
\end{figure}

\begin{remark}
	In cases where a level set function $\phi(\mathbf{x})$ describes the boundary $\Gamma = \{ \mathbf{x} \in \Omega \, | \, \phi(\mathbf{x}) = 0\}$, one may identify
	\begin{eqnarray}
		\mathbf{n} &=& \nabla \phi \\
		s &=& \phi(\mathbf{x}).
	\end{eqnarray}
Here we have assumed $|\nabla \phi| = 1$ and $\phi(\mathbf{x}) > 0$ represents the domain $\Omega_s$ while $\phi(\mathbf{x}) < 0$ corresponds to the domain $\Omega_p$.  
\myremarkend
\end{remark}

\item[Step 3.]
Construct the extension $\mathbf{\tilde{g}}$ using the functions $B_j(x)$ and the coordinates $(\boldsymbol \xi, s)$. 

For brevity, we introduce notation for the normal derivatives at the boundary $\Gamma$.
\begin{eqnarray}
	\mathbf{u}_n(\boldsymbol \xi) &=& (\mathbf{n}\cdot\nabla)\mathbf{u}|_{\Gamma} \\
	\mathbf{u}_{nn}(\boldsymbol \xi) &=& (\mathbf{n}\cdot\nabla)^2 \mathbf{u} |_{\Gamma}\\ \nonumber
	&\vdots& \\ 
	\mathbf{u}_{n_k}(\boldsymbol \xi) &=&  (\mathbf{n}\cdot\nabla)^k \mathbf{u} |_{\Gamma}
\end{eqnarray}
Again, we note that higher derivatives of $\mathbf{u}$ are discontinuous across the boundary. Therefore, $\mathbf{u}_{n_k}$ is evaluated as the limit approaching the boundary from the physical domain.  The extension is then
\begin{eqnarray} \nonumber
	\mathbf{\tilde{g}}(\mathbf{x}) &=& \big( \mathbf{g}(\boldsymbol \xi ) - \mathbf{G} \big)\,\, B_0\big( s \, l^{-1} \big) + l \,\, \mathbf{u}_n(\boldsymbol \xi ) \,\, B_1\big( s \, l^{-1} \big)  \\ \nonumber
&+& l^2 \,\, \mathbf{u}_{nn}(\boldsymbol \xi) \,\, B_2\big(s \, l^{-1} \big) + \ldots \\
&+& l^k \,\, \mathbf{u}_{n_k}(\boldsymbol \xi) \,\, B_k\big(s \, l^{-1} \big) + \mathbf{G}. \label{Extension}
\end{eqnarray}
Note that $\mathbf{\tilde{g}}$ decays to $\mathbf{G}$, i.e. $\mathbf{\tilde{g}} \rightarrow \mathbf{G}$, as $s \rightarrow l$.  Therefore $\mathbf{G}$ can be any time-dependent constant vector, however, for numerical purposes one should choose $\mathbf{G}$ close to the boundary average of $\mathbf{g}$
\begin{eqnarray}
	\mathbf{G} &=& A^{-1} \int_{\Gamma} \mathbf{g} \du A \\
	A &=& \int_{\Gamma} \du A.
\end{eqnarray}

\end{enumerate}

\begin{remark}
	Since values of $\mathbf{\tilde{g}}$ inside $\Omega_s$ depend on derivatives of $\mathbf{u}$ on the boundary, the function $\mathbf{\tilde{g}}$ described in (\ref{Extension}) depends linearly on $\mathbf{u}$.
\myremarkend
\end{remark}

\begin{remark}
	We can check that the construction (\ref{Extension}) satisfies the properties (P1)-(P3).  For $\mathbf{x} \in \Gamma$, we have $s = 0$, so that $\mathbf{\tilde{g}} = \mathbf{g}$, thereby satisfying (P1). To check higher derivatives, we first note that differentiating (\ref{ForwardCoordinates}) with respect to $s$ yields $\frac{\partial \mathbf{x}(\boldsymbol \xi, s)}{\partial s} = \mathbf{n}(\boldsymbol \xi)$.  As a result any function independent of $s$, $y(\boldsymbol \xi, s) = y(\boldsymbol \xi)$, has the property that
\begin{eqnarray} \label{DerivativeXi}
	(\mathbf{n}\cdot \nabla) y(\boldsymbol \xi) &=&  \sum_j \frac{\partial x_j}{\partial s} \frac{\partial y(\boldsymbol \xi)}{\partial x_j} \\
	&=& \frac{\partial y(\boldsymbol \xi)}{\partial s} \\
	&=& 0.
\end{eqnarray}
 Meanwhile, we also have 
\begin{eqnarray}
	(\mathbf{n}\cdot \nabla)^i B_j(s \, l^{-1}) |_{s = 0} &=& \Big(\frac{\partial}{\partial s} \Big)^i B_j(s \, l^{-1})  |_{s = 0}\\
\label{DerivativePhi}
	&=& l^{-i} \delta_{ij}
\end{eqnarray}
where $\delta_{ij}$ is the Kronecker delta.  Combining the two properties above, we have 
\begin{eqnarray}
	(\mathbf{n}\cdot \nabla)^i \mathbf{\tilde{g}} |_{\mathbf{x} \in \Gamma}&=& \mathbf{u}_{n_i}(\boldsymbol \xi).
\end{eqnarray}
Therefore, we recover properties (P2)-(P3).

\myremarkend
\end{remark}

\section{A Model equation} \label{Section_ModelEquation}
In this section we examine solutions to the steady-state heat equation to provide some explanation for how the extension function $\tilde{g}$ improves the analytic convergence rate of the penalized equations to the original problem.  In particular, we seek to quantify the error that results from the additional penalty forcing.  As a non-penalized problem, consider
\begin{eqnarray} \label{ModelEq}
		\partial_{xx} v &=& 0 \hspace{6mm} x \in [-1, 0] 
\end{eqnarray}
with boundary conditions: $v(-1) = 1$, $v(0) = 0$.  The solution is then $v(x) = -x$ for $x \in [-1, 0]$.

We note that solving explicit examples of the steady-state equations do not give general sharp convergence estimates, however, they do provide a rigorous lower bound on the convergence rate of the penalized equation to the exact non-penalized equation.
The equivalent one dimensional steady-state penalized problem is then
\begin{eqnarray} \label{ModelPenaltyEq}
		\partial_{xx} u &=& \eta^{-1}H(x) (u - \tilde{g}) \hspace{6mm} x \in [-1, \infty)
\end{eqnarray}
with boundary conditions $u(-1) = 1$, $u(\infty) = 0$.  Here $H(x)$ is the Heaviside function
\begin{eqnarray}
	H(x) = \begin{cases}
		0 & \text{for } x < 0 \\
		1 & \text{for } x \geq 0
	\end{cases}
\end{eqnarray}
We now examine the convergence of solutions $u \rightarrow v$ in the limit $\eta \rightarrow 0$ for different extensions $\tilde{g}$.

\begin{remark} \label{RmkDiscSecondDer}
	As a result of the discontinuous Heaviside function $H(x)$, the solution $u$ to equation (\ref{ModelPenaltyEq}) has one continuous derivative ($u \in C^1$). Higher derivatives are discontinuous across $x = 0$.  \myremarkend
\end{remark}
In light of remark (\ref{RmkDiscSecondDer}), we take $\tilde{g}$ to have derivatives matching $u$ from the physical domain $x = 0^-$ and not $x = 0^+$. 

\begin{proposition} \label{Proposition1}
	Suppose that $\tilde{g}$ is a bounded $C^{k+1}$ function that matches $k$ derivatives of $u$ at $x = 0^-$. Namely
	\begin{enumerate}
   	 \item $\tilde{g}(0) = 0$
	    \item $\tilde{g}'(0) = u'(0)$
   	 \item  $\tilde{g}^{(m)}(0) = u^{(m)}(0^-) = 0$ for $2 \leq m \leq k$.  
	\end{enumerate} 
	Then the solution $u$ converges to $v$ as
	\begin{eqnarray}
		\max_{x \in [-1, 0]} |u - v| &=& O(\eta^{(k+1)/2})
	\end{eqnarray}
\end{proposition}

\begin{proof}

In the region $-1 \leq x \leq 0$, $u$ has the solution
\begin{eqnarray}
	u = (1 + c) + c x
\end{eqnarray}
for some constant $c$.  To construct the solution on $x \geq 0$, we note that one may write $\tilde{g}$ as 
\begin{eqnarray} \label{Localg}
	\tilde{g}(x) = c x + x^{k + 1} R(x)
\end{eqnarray}
for some remainder function $R(x)$, where in general $R(0) \neq 0$.  By construction (\ref{Localg}) matches the first $k$ derivatives of $u$ at $x = 0^-$.  In addition, we assume that $\tilde{g}(x)$ and $\tilde{g}'(x)$ are bounded, so that $R(x)$ and $R'(x)$ are also bounded functions. 
On $x \geq 0$, $u$ then solves
\begin{eqnarray}
	\partial_{xx} u - \eta^{-1} u = -\eta^{-1}\big(c x + x^{k + 1} R(x)\big)
\end{eqnarray}
To obtain the correct scaling, we rescale $x = \eta^{1/2}X$ to obtain
\begin{eqnarray}
	\partial_{XX} u - u = -c \eta^{1/2} X + \eta^{(k+1)/2} X^{k + 1} R(\eta^{1/2} X)
\end{eqnarray}
The general solution is then
\begin{eqnarray}
	u(X) = \begin{cases}
		(1 + c) + c \eta^{1/2} X & \text{if } X < 0 \\
		b e^{-X} + c \eta^{1/2} X + \eta^{(k+1)/2} Q(X) & \text{if } X \geq 0
	\end{cases}
\end{eqnarray}
where we have excluded the term $e^{X}$ since it diverges as $X \rightarrow \infty$.  In addition, $Q(X)$ is a particular solution (which stays bounded as $X \rightarrow \infty$ ) to
\begin{eqnarray}
	Q_{XX} - Q =  X^{k + 1} R(\eta^{1/2} X)
\end{eqnarray}
For instance, one may write a particular solution as
\begin{eqnarray}
	Q(X) &=& \frac{1}{2}\int_{0}^{X} (e^{X - y} - e^{-X + y}) y^{k+1}R(\eta^{1/2} y) \du y - A e^{X} \\
	A &=& \frac{1}{2} \int_0^{\infty} e^{-y} y^{k+1} R(\eta^{1/2} y) \du y
\end{eqnarray}
Letting $R_m = \max_{y} |R(y)|$, we also have the bound
\begin{eqnarray}
	|Q(0)| = |A| &\leq& \frac{R_m}{2} \int_0^{\infty} y^{k+1}e^{-y} \du y = Q_0.
\end{eqnarray}
The same type of argument holds for bounding $|Q'(0)| \leq Q_1$. 

To solve for the unknown constants, $c$ and $b$, we use the fact that $u$ and $u'$ are continuous across $x = 0$.  We therefore obtain the two equations
\begin{eqnarray}
	1 + c &=& b + \eta^{(k+1)/2} Q(0) \\
	\eta^{1/2} c &=& -b + \eta^{1/2}c + \eta^{(k+1)/2} Q'(0)
\end{eqnarray}
Upon solving for $b$ and $c$, the error between $u$ and $v$ on the physical domain $-1 \leq x \leq 0$ is
\begin{eqnarray}
	\max_{x \in [-1, 0]} |u - v| &=& 1 + c\\
	&=& (Q(0) + Q'(0) ) \eta^{(k+1)/2} \\
	&\leq& (Q_0 + Q_1) \eta^{(k+1)/2} \\	
	&=& O(\eta^{(k+1)/2})
\end{eqnarray}
Hence, for the model problem, matching $k$ derivatives of $\tilde{g}$ yields a convergence rate of $(k+1)/2$.  In particular, when $k = 0$, we recover the known convergence rate $\eta^{1/2}$ of the standard penalty method.
\end{proof}

\begin{remark}
	 Using higher derivatives in the construction of $\tilde{g}$ which are taken as limits from the domain $x \searrow 0^+$, does not yield the convergence rate stated in proposition (\ref{Proposition1}).  As an example, we take $\tilde{g}^+ = u'(0) B_1(x) + u''(0^+) B_2(x)$ where 
\begin{eqnarray}
	B_1(x) &=& \frac{5}{2}e^{-x} - 4 e^{-2x} + \frac{3}{2}e^{-3x} \\
	B_2(x) &=& \frac{1}{2}e^{-x} - e^{-2x} + \frac{1}{2}e^{-3x}.
\end{eqnarray}
For such a $\tilde{g}^+$, the solution $u$ to problem (\ref{ModelPenaltyEq}) yields only a first order error
\begin{eqnarray}
	\max_{[-1,0]}|u - v| \sim \frac{11}{6}\eta.
\end{eqnarray}
In contrast, taking $\tilde{g}^- = u'(0) B_1(x) + u''(0^-) B_2(x)$ yields a convergence rate in agreement with (\ref{Proposition1})
\begin{eqnarray}
	\max_{[-1,0]}|u - v| \sim 11 \eta^{3/2}.
\end{eqnarray}
\myremarkend
\end{remark}

\section{Stability} \label{Section_Stability}
In this section we establish numerical stability criteria for the $1D$ penalized heat equation.  To examine stability, we work with the domain $\Omega_p = (0, \pi)$, $\Omega_s = [\pi, 2\pi]$ and periodic boundary conditions.  Moreover, we take $g(0) = g(\pi) = 0$ to capture a $u = 0$ boundary condition at the fluid-solid boundary. A simple Euler scheme matching one derivative of $u$ at the interface is then 
\begin{eqnarray} \label{SimpleEuler}
	u^{n+1} &=& \big(I + \Delta t \; \Delta \big) u^n - \Delta t \; \eta^{-1} \chi_s \; (u^n - \tilde{g}^n) \\ \label{SimpleEuler2}
	\tilde{g}^n &=& u_x^n(\pi) B_1(x - \pi) - u_x^n(2\pi) B_1(2\pi-x)
\end{eqnarray}
In general, adding derivatives of $u$ to $\tilde{g}(x)$ can reduce, by factors of $h$, the time step restriction for an explicit scheme.  In the case at hand, however, the structure of $\tilde{g}(x)$ results in the same time step restriction as the original volume penalty method, namely
\begin{eqnarray} \label{TimeStepRestriction}
	\Delta t < \min \{ O(h^2), O(\eta) \}.
\end{eqnarray}
Here $h$ is either the grid spacing of a finite difference scheme, or alternatively $h^{-1}$ scales as the largest wavenumber in a Fourier method.

We note that although (\ref{SimpleEuler}) is a linear recursion relation, the right hand side is not a normal operator.  As a result, a rigorous proof of (\ref{TimeStepRestriction}) requires bounding the eigenvalues for the spatially discrete system (\ref{SimpleEuler}).  The analysis is further complicated by the fact that the operators (or matrices) on the right hand side of (\ref{SimpleEuler}) do not commute.

In this section we establish the time step restriction (\ref{TimeStepRestriction}). To do so, we first compute the eigenvalues for the penalty term using a finite difference scheme.  We show that although the penalty term contains derivatives of $u$, the eigenvalues remain $O(\eta^{-1})$ and do not become $O(\eta^{-1} h^{-1})$. 

Secondly, to show that the addition of the Laplacian does not alter the restriction (\ref{TimeStepRestriction}), we numerically compute the eigenvalues for equation (\ref{SimpleEuler}) using a finite difference scheme.

\subsection{Eigenvalues of the penalty term (Finite differences)}

In practice, one does not observe the time step restriction governed by the norm of $\tilde{g}$, but rather the larger bound in (\ref{TimeStepRestriction}).  Here we provide a stability criteria by analytically computing the penalty term eigenvalues for a finite difference scheme.

Let $x_k = k h$ for $0 \leq k \leq N - 1$ with grid spacing $h = 2\pi/N$.  Furthermore, denote the discrete vector $\mathbf{u} = [u(x_0) \, u(x_1) \ldots u(x_{N-1})]^T$. 

We are then interested in evaluating the eigenvalues of the penalty term 
\begin{eqnarray}
	\mathbf{B} \mathbf{u} &=& \lambda \mathbf{u} \\
	\mathbf{B} &=& - \eta^{-1}( \mathbf{I}_\chi - \mathbf{v}_1 \mathbf{d}_1^T - \mathbf{v}_2 \mathbf{d}_2^T)
\end{eqnarray}
where $\mathbf{B}$ is the finite difference matrix corresponding to the penalty term. Here $\mathbf{I}_\chi$ is the identity matrix restricted to $x \in \Omega_s$ while $\mathbf{v}_1$ and $\mathbf{v}_2$, are vectors with components
\begin{eqnarray}
	(\mathbf{v}_1)_k &=& \chi_s(x_k) B_1(x_k - \pi) \\
	(\mathbf{v}_2)_k &=& -\chi_s(x_k) B_1(2\pi - x_k)
\end{eqnarray}
In addition, $\mathbf{d}_1$ and $\mathbf{d}_2$ are column vectors which approximate the derivatives of a vector $\mathbf{u}$ as
\begin{eqnarray}
	u_x(\pi) \approx \mathbf{d}_1^T \mathbf{u} \\
	u_x(2\pi) \approx \mathbf{d}_2^T \mathbf{u}
\end{eqnarray}
For instance, a centered difference approximation to the derivative $u_x(2\pi)$ would have $(\mathbf{d}_2)_{N-1} = -(2h)^{-1}$, $(\mathbf{d}_2)_1 = (2h)^{-1}$ and $(\mathbf{d}_2)_k = 0$ for $k = 0$ and $1 < k < N-1$.  Lastly, since the support of $B_1(x)$ is restricted to $x < 1$, the function $B_1(x - \pi) = 0$ for $x > \pi + 1$.  Hence, the numerical derivative of $B_1(x - \pi)$ at $x = 2\pi$ is zero (or similarly with $B_1(2\pi - x)$ at $x = 0$) 
\begin{eqnarray} \label{DerivativeVanish1}
	\mathbf{d}_2^T \mathbf{v}_1 &=& 0 \\ \label{DerivativeVanish2}
	\mathbf{d}_1^T \mathbf{v}_2 &=& 0.
\end{eqnarray}
Combining the orthogonality conditions (\ref{DerivativeVanish1})--(\ref{DerivativeVanish2}) with the fact that $\mathbf{I}_\chi \mathbf{v}_{1, 2} = \mathbf{v}_{1, 2}$, implies that $\mathbf{u} = \mathbf{v}_{1, 2}$ are eigenvectors with corresponding eigenvalues
\begin{eqnarray} \label{DiscreteEigenvalues1}
	\lambda_{1} &=& -\eta^{-1} (1 - \mathbf{d}_{1}^T \mathbf{v}_1) \\ \label{DiscreteEigenvalues2}
	\lambda_{2} &=& -\eta^{-1} (1 - \mathbf{d}_{2}^T \mathbf{v}_2). 
\end{eqnarray} 
All other eigenvalues of $\mathbf{B}$ then lie in the space perpendicular to $\mathbf{v}_1$ and $\mathbf{v}_2$ resulting in either $\lambda = 0$ or $\lambda = -\eta^{-1}$. The eigenvalues (\ref{DiscreteEigenvalues1})--(\ref{DiscreteEigenvalues2}) are therefore directly a result of the modified penalty term and depend specifically on the component values of $\mathbf{d}_{1,2}$.  As a result, the products $\mathbf{d}_{1 ,2}^T \mathbf{v}_{1, 2} \in (0, 1]$ depending on how one builds the numerical derivative vector $\mathbf{d}_{1, 2}$.  

As an example, taking a centered difference approximation to the derivative $u_x(2\pi)$ yields
\begin{eqnarray}
	\mathbf{b}_2^T \mathbf{v}_2 &=& \frac{1}{2h}( (\mathbf{v}_2)_1 - (\mathbf{v}_2)_{N-1} ) \\
&\approx& 0.5.
\end{eqnarray}
The second line follows since $(\mathbf{v}_2)_{N-1} = 0$ while $(\mathbf{v}_2)_1 \approx h$ because the function $B_1'(0) = 1$. 

In general, the product  $\mathbf{d}_{1 ,2}^T \mathbf{v}_{1, 2}$ will be a weighted average of the derivatives of $B_1(x)$ on the left and right of the interface.  As a result, all eigenvalues $\lambda$ of $\mathbf{B}$ satisfy $-\eta^{-1} \leq \lambda \leq 0$. Therefore, modifying the penalty term does not change the time step restriction $\Delta t < 2 \eta$ for a simple Euler scheme $\mathbf{u}^{n+1} = \mathbf{u}^n + \Delta t \; \mathbf{B}\mathbf{u}^n$.

\subsection{Numerical eigenvalues}

In the follow section we numerically compute the eigenvalues of (\ref{SimpleEuler})--(\ref{SimpleEuler2}) using a finite difference scheme for the spatial derivatives\footnote{Although not shown, a similar result of $\Delta t < \min \{  N^{-2}, 1.1 \eta \}$ holds for a Fourier scheme.}.  The scheme then has the form 
\begin{eqnarray} \label{DiscreteLinearSystem}
	\mathbf{u}^{n+1} &=& \big[ \mathbf{I} + \Delta t \; (\mathbf{L} + \mathbf{B}) \big] \mathbf{u}^n 
\end{eqnarray}
where $\mathbf{L}$ is the standard 3-point stencil discrete Laplacian.  As a result, the eigenvalues of the linear system (\ref{DiscreteLinearSystem}) approach the real values associated with the Laplacian $\Delta$ when $\eta \rightarrow \infty$, and the values associated with the penalty term when $\eta \rightarrow 0$.

To compute the eigenvalues numerically, we fix a grid with $N$ points ($256 \leq N \leq 4096$) and examine the range $10^{-9} \leq \eta \leq 1$.  
\begin{proposition} \label{Proposition3}
	(Practical stability) In practice, the numerical scheme (\ref{DiscreteLinearSystem}) is stable provided one takes the time step restriction
	\begin{eqnarray} \label{StabilityCrit3}
		\Delta t < \min \{ 0.5 h^2, 1.2 \eta \}.
	\end{eqnarray}
\end{proposition} 

\begin{remark}
	The exact constant $1.2$ in (\ref{StabilityCrit3}) depends on numerical details such as how one interpolates derivatives to the interface or the nature of the functions $B_0(x), B_1(x) \ldots$. \myremarkend
\end{remark}
Here figure $\ref{Eigenvalues}$ shows that the numerical eigenvalues for $N = 2048$ and $\eta = 10^{-7}$ are stable with a time step restriction (\ref{StabilityCrit3}).

\begin{figure}[htb!]
	\centering
    \includegraphics[width = \textwidth]{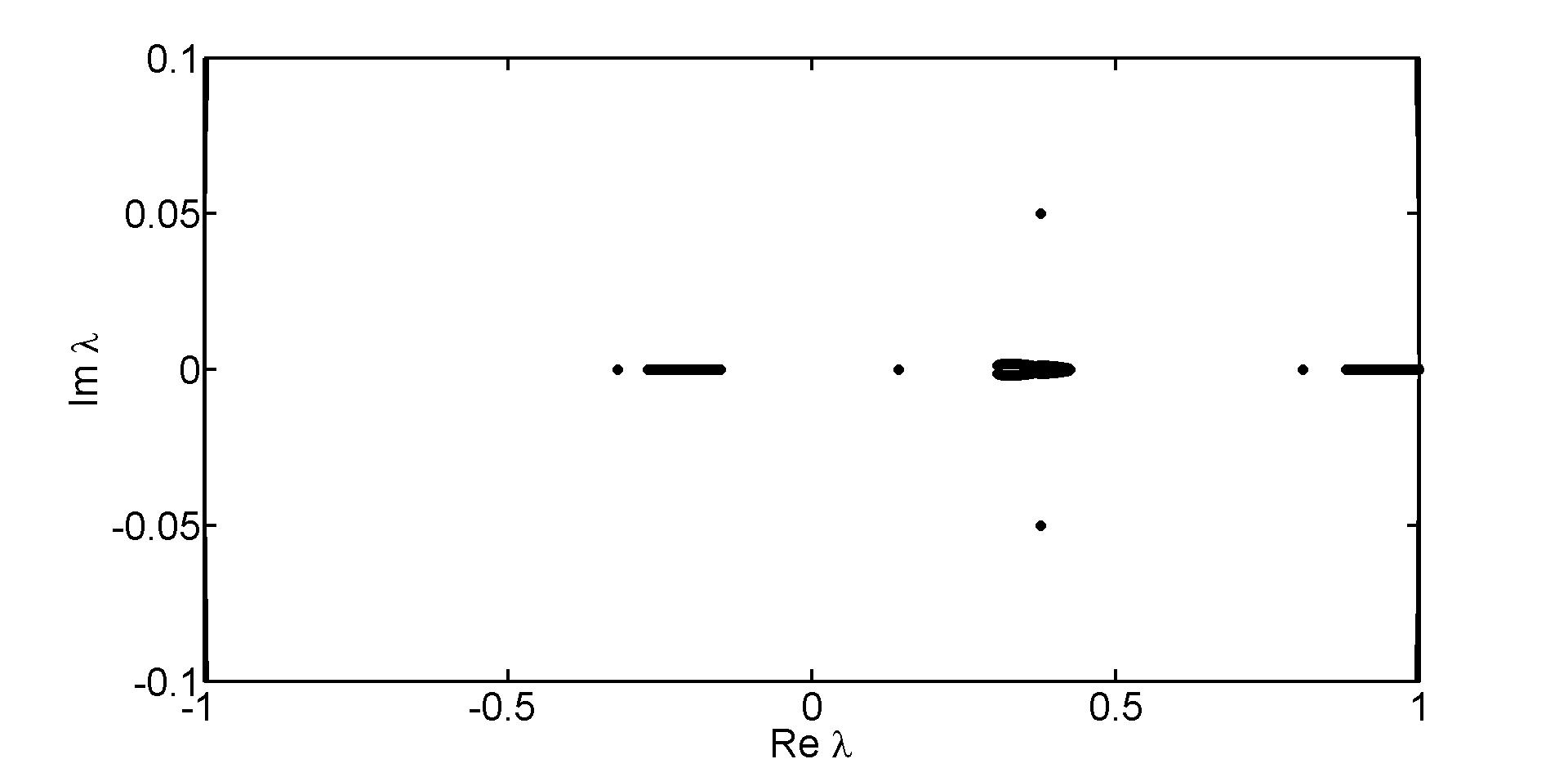} \\
    \caption{Scatter plot of the numerical eigenvalues for (\ref{DiscreteLinearSystem}). Here $N = 2048$, $\eta = 10^{-7}$ and $\Delta t$ is taken from (\ref{StabilityCrit3}).}
 \label{Eigenvalues}
\end{figure}

\section{Numerical example: Heat Equation} \label{Section_NumericsHeatEq}

In the following section we provide numerical examples for the heat equation in dimension $D = 1, 2$.  Specifically, we combine the analytic convergence and stability results from the previous sections to show that one may achieve high order numerical schemes.  As a starting point, we demonstrate high order convergence in $D = 1$ dimension.  We then move to $D = 2$, and outline additional details that arise from the numerical construction of the extension $\tilde{g}(\mathbf{x})$. 

\subsection{1D Heat Equation}

To test the convergence rates for the penalized heat equation, we use a manufactured solution approach.  We note that the forced heat equation on $x \in [0, 2\pi]$, 
\begin{eqnarray} \label{ForcedHeat}
	\partial_t u &=& u_{xx} + f \\ \label{ExternalForcing1}
	f &=& e^{\sin(x+t)}\big[ \cos(x+t) + \sin(x+t) - \cos^2(x+t)\big] \\
	u(x, 0) &=& e^{\sin(x)}
\end{eqnarray}
has an exact solution $u_e = e^{\sin(x+t)}$.  To quantify the total error, we penalize equation (\ref{ForcedHeat}) as
\begin{eqnarray} \label{ForcedPenalty1}
	\partial_t u_{\eta} &=& (u_{\eta})_{xx} + f - \eta^{-1}\chi_s \; (u_{\eta} - \tilde{g}) 
\end{eqnarray}
and compare the numerical solution of (\ref{ForcedPenalty1}) to the exact one from (\ref{ForcedHeat})\footnote{One can also restrict the forcing $\tilde{f} = f (1-\chi_s)$ to the physical domain, and obtain similar results.}.

To discretize in space, we use an equispaced grid with fourth order stencils for all derivatives.  In addition, we treat all terms explicitly in time with a second order (improved) Euler scheme.  When constructing the extension $\tilde{g}$, we first compute the derivatives of $u$ at each grid point, i.e. $u_x(x_k)$ or $u_{xx}(x_k)$.  We then interpolate the values of $u_x$ and $u_{xx}$ from the regular grid points to the points $x_{\Gamma}$ on the interface.

\begin{remark}
	The solution to the penalized heat equation $u$ has a discontinuous second derivative $u_{xx}$ across the interface.  As a result, interpolating $u_{xx}$ using regular grid points on both sides of the interface will produce a weighted average of right and left derivatives $u_{xx}(0^-)$ and $u_{xx}(0^+)$ in the construction of $\tilde{g}$.  We note that in practice, such a procedure does not appear to alter the final numerical convergence rate.
	\myremarkend 
\end{remark}

For our tests, we choose a solid region centered at $\pi$ to be $\Omega_s = [\pi - 0.7, \pi + 0.7]$.  To satisfy the stability restriction, we then take $\Delta t = 0.2 \; h^2$, $h = 2\pi/N$ and slave $\eta = 5 \; \Delta t$ so that all parameter values are fixed by the number of grid points $N$.  For each $N = 2^k$, with $6 \leq k \leq  12$, we then numerically integrate (\ref{ForcedPenalty1}) up to a final time of $T = 1$.  We repeat the procedure using $0, 1$ and $2$ derivatives of $u$ in constructing the extension $\tilde{g}$ and compare the numerical solution to the exact one (i.e. that of the unpenalized problem).  Here figure \ref{HeatConvergence1D} shows the convergence rates for matching different derivatives, while figures \ref{ForcedSolution} and \ref{ErrorField1D} show a typical solution and the corresponding error respectively.

\begin{figure}[htb!]
	\centering
    \includegraphics[width = \textwidth]{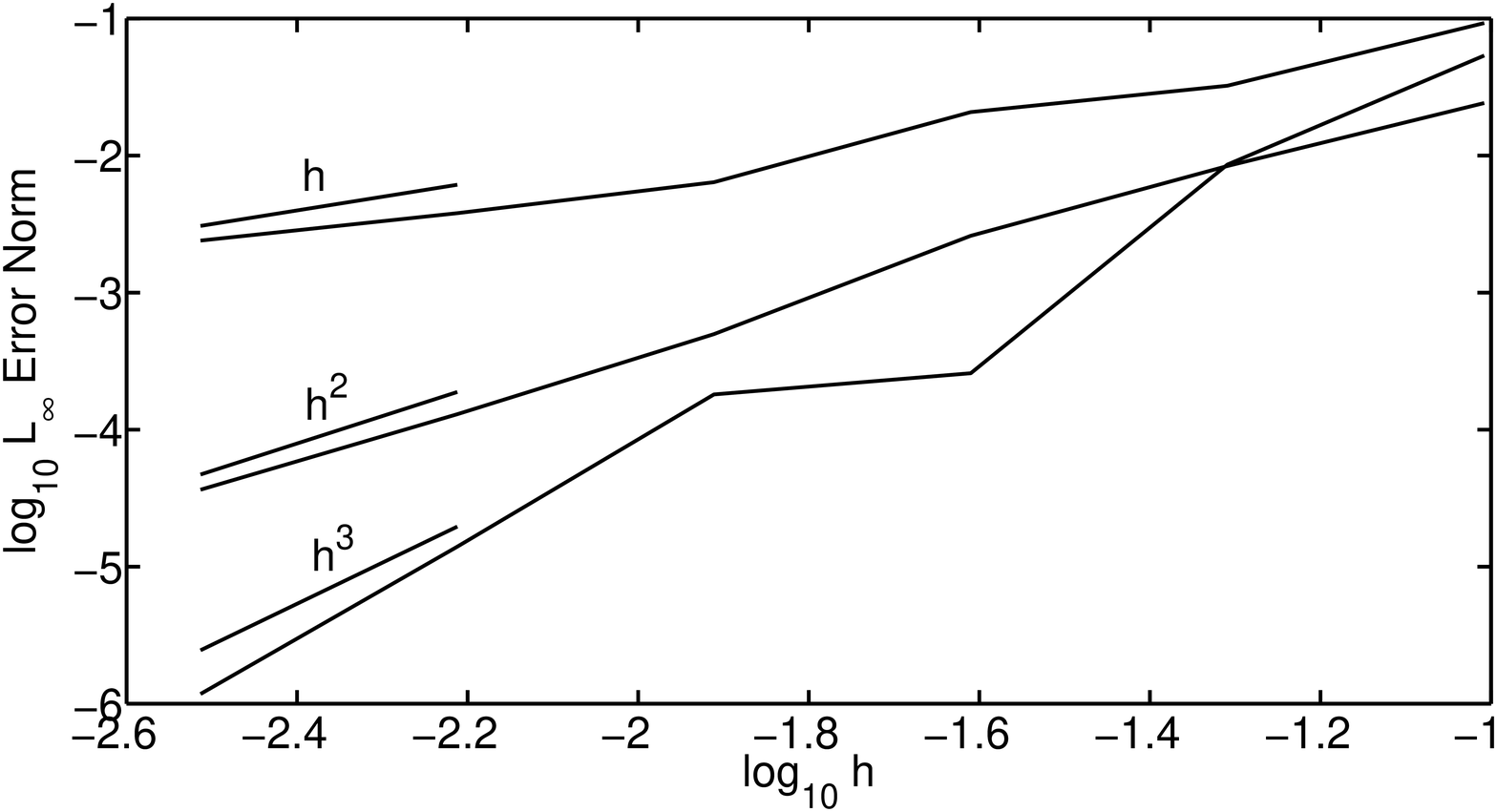} \\
    \caption{Plot of numerical errors $||u - u_{\eta, num}||_{\infty}$ for different values of $N$. The three curves correspond to building $\tilde{g}$ using $0, 1, 2$ derivatives of $u$ and result with convergence rates of $1, 2, 3$ respectively.}
 \label{HeatConvergence1D}
\end{figure}
 
\begin{figure}[htb!]
	\centering
    \includegraphics[width = \textwidth]{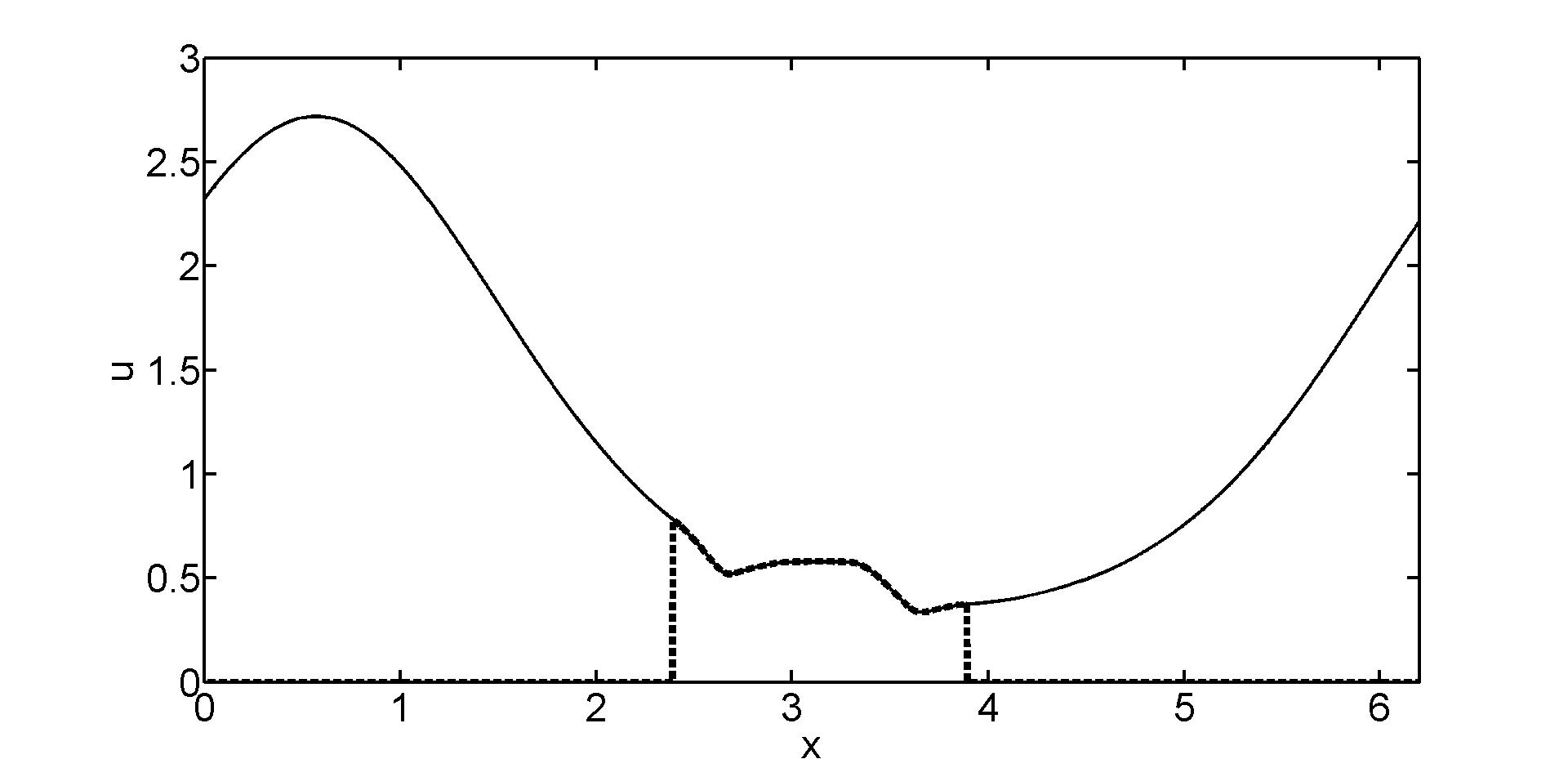} \\
    \caption{Plot of the numerical solution $u_{\eta, num}$ (thin line) with the extension $\tilde{g}$ (dashed line) for $N = 2048$. Here $\Omega_s = [\pi - 0.7, \pi + 0.7]$ is the solid domain}
 \label{ForcedSolution}
\end{figure}

\begin{figure}[htb!]
	\centering
    \includegraphics[width = \textwidth]{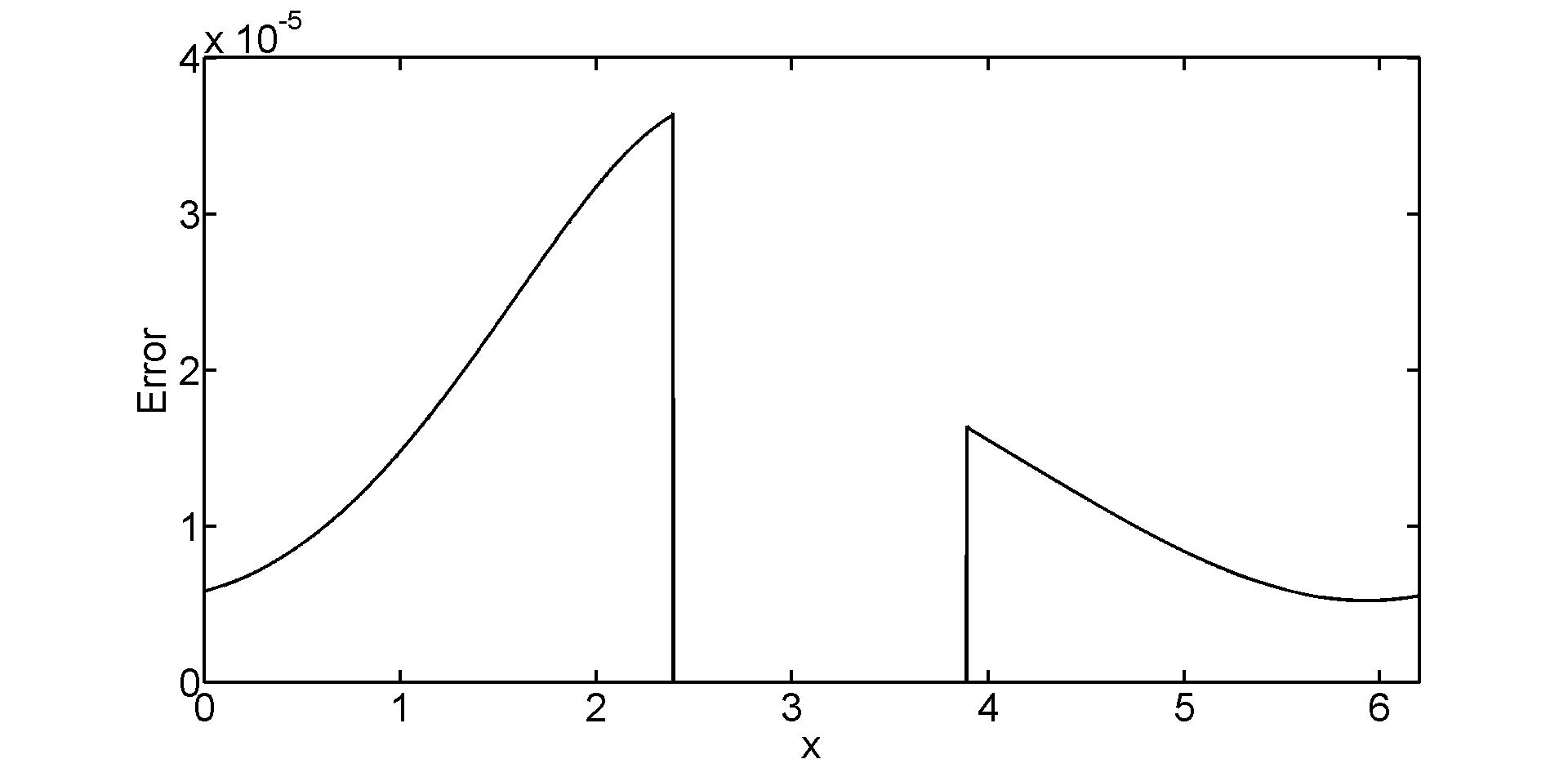} \\
    \caption{Plot of the total error in solving the penalized equations matching $1$ derivative of $u$ in the extension.  Here, $N = 2048$, $T = 1$ is the integration time, while $\Omega_s = [\pi - 0.7, \pi + 0.7]$ is the solid domain.}
 \label{ErrorField1D}
\end{figure}

\subsection{2D Heat Equation}

In the following subsection, we outline the numerical details for a $D = 2$ scheme.  Here we work with an equispaced, regular grid with $N \times N$ points ($64 \leq N \leq 512$), and immerse the boundary $\Gamma$.  The main difference when moving to higher dimensions is how one computes the extension $\tilde{g}(\mathbf{x})$.  To illustrate the construction, we refer to figure \ref{InterpolationGrid}.  To build $\tilde{g}(\mathbf{x})$ we first compute all appropriate derivatives at each grid point (both on $\Omega_s$ and $\Omega_p$).  For each grid point $\mathbf{x} \in \Omega_s$ within a distance $l$ of $\Gamma$, we compute $\boldsymbol\xi(\mathbf{x})$ as the orthogonal projection of $\mathbf{x}$ onto $\Gamma$ and $s(\mathbf{x}) = ||\boldsymbol\xi(\mathbf{x}) - \mathbf{x}||_2$.  Using a regular 9 point stencil, we then perform a polynomial interpolation of all required derivatives from the grid points to $\boldsymbol\xi$.  Using the interpolated derivatives at $\boldsymbol\xi$, one can then compute the normal derivatives of $\mathbf{u}$ required in equation (\ref{Extension}) to construct $\tilde{g}(\mathbf{x})$ at each grid-point inside $\Omega_s$.  Figure \ref{Gtilde2D} illustrates a typical construction of $\tilde{g}(\mathbf{x})$. 

\begin{remark}
	For computational efficiency, one can precompute and store the values of $\boldsymbol\xi$ as well as the appropriate coefficients required to extrapolate derivatives to the interface $\Gamma$.
	\myremarkend
\end{remark}

As an example in $D = 2$, we take the computational domain $\Omega$ to be a periodic square with side length $2\pi$. For the penalized domain $\Omega_s$, we take a circle of radius $r = 1/2$ and center $(x_c, y_c) = (\pi, \pi)$.  The physical domain is then $\Omega_p = \Omega \setminus \Omega_s$.  To perform convergence tests, we again use a manufactured solution where $u_e = [e^{\sin(x)} + cos(y) ] \cos(t)$.  Here we perform a convergence test for the penalty parameter $\eta$.  To compute the convergence rate, we fix $N = 512$ and vary $5\times 10^{-5} \leq \eta \leq 10^{-1}$,  so that discrete numerical errors are smaller than the $\eta$-dependent error obtained by introducing the penalty term.  For different values of $\eta$, we then integrate the penalty equation for a time $T = 0.1$ and compute the error.  Figure \ref{ConvergencePlot2DHeat} shows the $L^{\infty}$ error between the penalized equation and the exact heat equation as a function of $\eta$.

\begin{figure}[htb!]
	\centering
    \includegraphics[width = \textwidth]{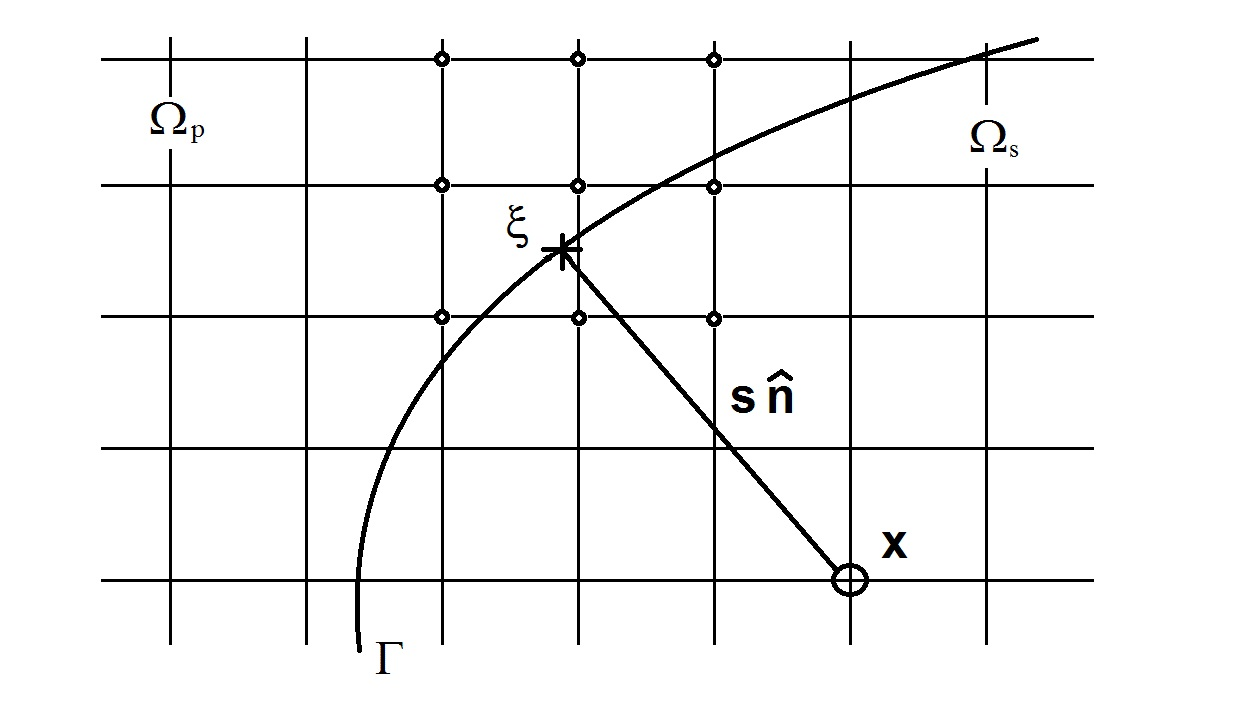} \\
    \caption{Regular grid with interpolation points.}
 \label{InterpolationGrid}
\end{figure}

\begin{figure}[htb!]
	\centering
    \includegraphics[width = \textwidth]{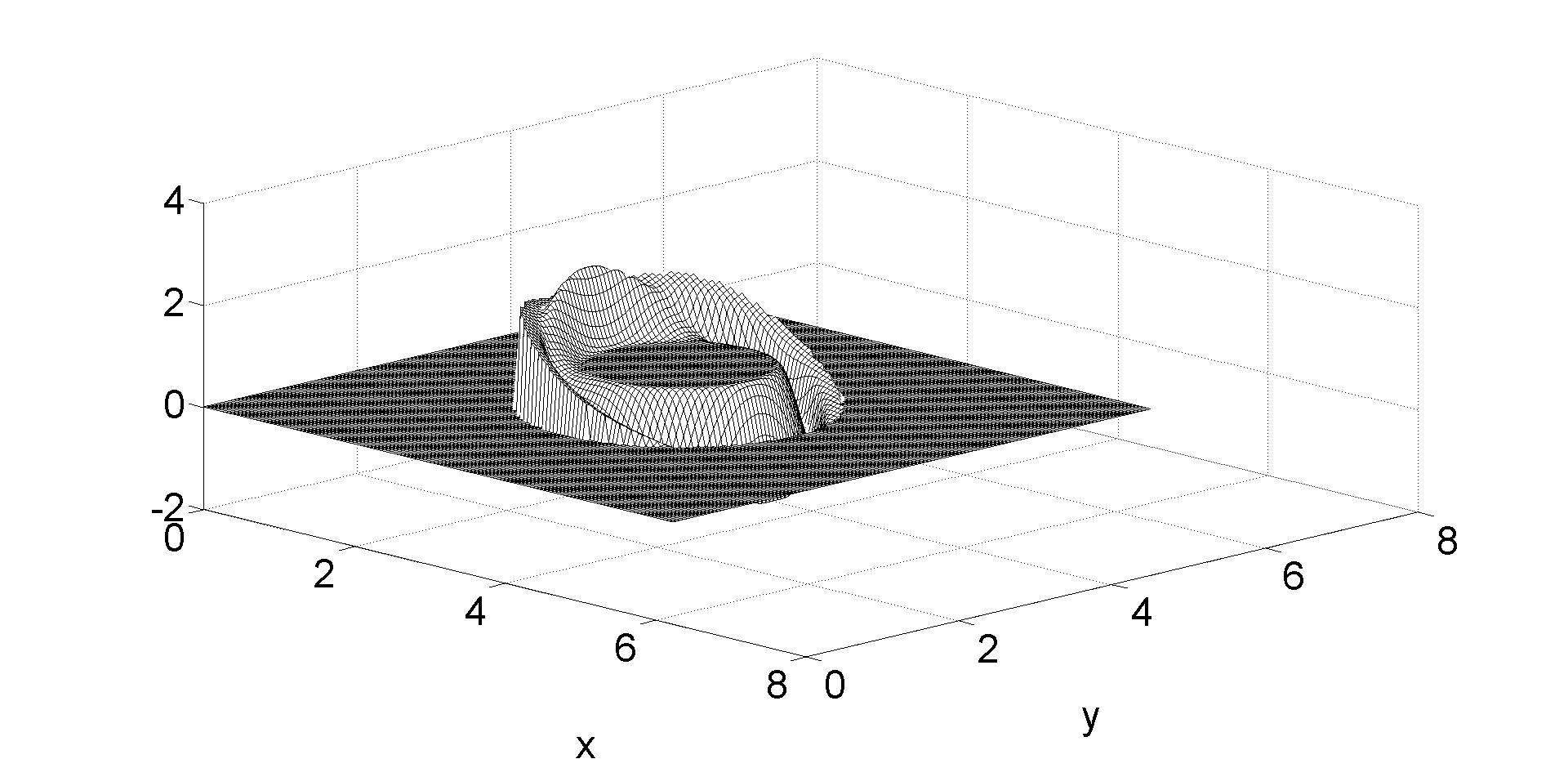} \\
    \caption{A sample 2D plot of $\tilde{g}$ matching $2$ normal derivatives of $u(\mathbf{x})$.  The plot is taken at $t = 0$ for the heat equation tests.}
 \label{Gtilde2D}
\end{figure}

\begin{figure}[htb!]
	\centering
    \includegraphics[width = \textwidth]{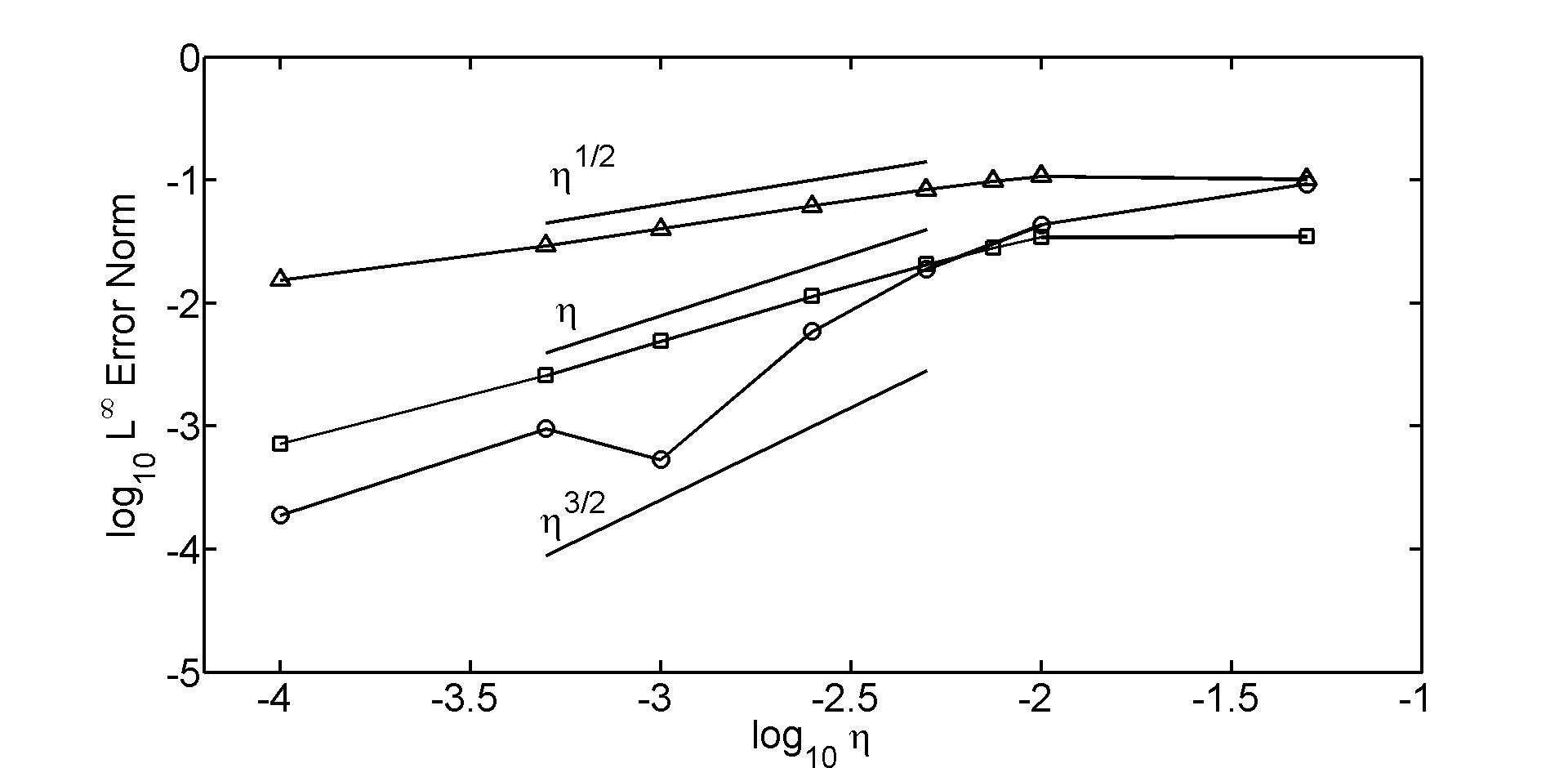} \\
    \caption{Converence plot of the $L^{\infty}$ error at $T  = 0.1$ for the heat equation tests. The plot shows curves for matching $0$ derivatives (triangles), $1$ derivative (squares), and $2$ derivatives (circles).  The straight lines compare the expected convergence rates of $O(\eta^{0.5})$, $O(\eta)$ and $O(\eta^{1.5})$ respectively.}
 \label{ConvergencePlot2DHeat}
\end{figure}

\section{Numerical example: 2D incompressible Navier-Stokes} \label{Section_NumericsNSEq}

The primary difficulty when transitioning from a penalized heat equation to the penalized incompressible Navier-Stokes equations is the addition of the velocity divergence constraint.  Other differences, such as moving from a scalar to a vector equation, or adding a nonlinear convective term do not pose new additional challenges to the penalized equations.  Intuitively, the difficulty with the divergence can be outlined as follows.  For the penalized heat equation, the active penalty term forces the function $u$ to closely track the extension function $\tilde{g}$.  When moving to a set of vector equations, the velocity vector $\mathbf{u}_{\eta}$ will closely track the term $\mathbf{\tilde{g}}$ inside the penalty region $\Omega_s$.  However, the component-wise construction of $\mathbf{\tilde{g}}$ will in general be such that $\nabla \cdot \mathbf{\tilde{g}} \neq 0$.  Consequently, to remain consistent, one should not force $\nabla \cdot \mathbf{u}_{\eta} = 0$ inside $\Omega_s$ but rather allow $\nabla \cdot \mathbf{u}_{\eta}$ to \emph{loosely} track $\nabla \cdot \mathbf{\tilde{g}}$. 

One approach for handling the divergence constraint is to replace $\nabla \cdot \mathbf{u}_{\eta} = 0$ with a Pressure Poisson Equation (PPE) \cite{Henshaw1994,JohnstonLiu2004,Shirokoff2011,ShirokoffRosales2011}.  Such an approach can provide a consistent method to compute the pressure and obtain high order schemes.  Since a PPE approach requires the additional solution of a Poisson equation with Neumann boundary conditions, we defer the implementation to future work. In our case, we utilize a projection method where we project the velocity divergence to zero inside the fluid domain.  We now discretize equations (\ref{ImprovedPenalty_Eq1})-(\ref{ImprovedPenalty_Eq2}) in time.

\subsection{Discretization in time}

Here we outline a pseudo-spectral scheme for solving the Navier-Stokes equations. For a second order scheme in $h$, we take a first order discretization in time with a time step restriction of the form outlined in (\ref{StabilityCrit3}).  Since the domain is $2\pi$ periodic, we can use the Fourier transform to invert the Poisson equation.  In the following algorithm we take a regular $N\times N$ grid.  We also denote the discrete Fourier transform by $\mathcal{F}$ so that $\hat{p}^{n}(\mathbf{k}) = \mathcal{F}[p^n]$ with $\mathbf{k} = (k_x, k_y)$ and $k = |\mathbf{k}|$.

\begin{algorithm} \label{Algorithm1} 
	(Navier-Stokes)
	\begin{enumerate}
		\item Given the velocity $\mathbf{u}_{\eta}^n$, compute an intermediate velocity $\mathbf{\tilde{u}}_{\eta}^{n+1}$
	\begin{eqnarray}
	\frac{\mathbf{\tilde{u}}_{\eta}^{n+1} - \mathbf{u}_{\eta}^n}{\Delta t} &=& \mathbf{F}^{n} - \frac{1}{\eta} \chi_s(\mathbf{x})(\mathbf{u}_{\eta}^n - \mathbf{\tilde{g}}^n) \\
		\mathbf{F}^n &=& -\mathbf{u}_{\eta}^n\cdot\nabla \mathbf{u}_{\eta}^n + \mu \Delta \mathbf{u}_{\eta}^n + \mathbf{f}^n
\end{eqnarray}

		\item Compute the pressure
		\begin{eqnarray} \label{DiscretePoisson}
			\Delta p_{\eta}^{n+1} &=& \frac{1}{\Delta t} ( \nabla\cdot \mathbf{\tilde{u}}_{\eta}^{n+1}) (1 - \chi_s) - \mathcal{A}.
		\end{eqnarray}
		For $k = 0$ set $\hat{p}_{\eta}^{n+1}(0) = 0$, while for $k \neq 0$	take	
		\begin{eqnarray} \label{DiscretePoisson2}
			\hat{p}_{\eta}^{n}(\mathbf{k}) &=& -\frac{1}{k^2} \mathcal{F}[  \frac{1}{\Delta t} ( \nabla\cdot \mathbf{\tilde{u}}_{\eta}^{n+1}) (1 - \chi_s)].
		\end{eqnarray}
		Note that $\mathcal{A}$ does not appear in the Fourier transform and at no time does one ever compute $\mathcal{A}$. The value of $\mathcal{A}$ is hidden as a consistency condition in setting $\hat{p}_{\eta}^{n+1}(0) = 0$. 
		\item Update the velocity $\mathbf{u}_{\eta}^{n+1}$
			\begin{eqnarray}
				\mathbf{u}_{\eta}^{n+1} &=& \mathbf{\tilde{u}}_{\eta}^{n+1} - (\Delta t)\mathcal{F}^{-1} [-\imath \mathbf{k} \hat{p}_{\eta}^{n}].
			\end{eqnarray}
	\end{enumerate}
	Note that either Fourier transforms or a second order finite difference scheme can be used when computing the derivatives $\partial_{x_j} \mathbf{u}_{\eta}$ in algorithm (\ref{Algorithm1}).  

Since the second derivatives are discontinuous, we compute $\Delta \mathbf{u}_{\eta}$ using finite differences.  
\end{algorithm}

In the Poisson equation for the pressure, $p_{\eta}$ is only determined up to a constant. To uniquely determine $p_{\eta}$ we enforce $\int_{\Omega} p_{\eta} \du V = 0$. Meanwhile the value of $\mathcal{A}$ is chosen so that the Poisson equation satisfies the standard solvability condition.  Namely, 
\begin{eqnarray} \label{SolvabilityCondition}
	\mathcal{A} &=& \frac{\Delta t^{-1}}{V} \int_{\Omega_p} \nabla \cdot \mathbf{\tilde{u}}_{\eta} \du V \\
					&=& \frac{\Delta t^{-1}}{V} \int_{\Gamma} \mathbf{\tilde{u}}_{\eta}\cdot \mathbf{n} \du A \\
					&=& \frac{\Delta t^{-1}}{V} \int_{\Gamma} (\mathbf{\tilde{u}}_{\eta} - \mathbf{g})\cdot\mathbf{n} \du A 
\end{eqnarray}
where in the last line we have used the fact that $\int_{\Gamma} \mathbf{g}\cdot \mathbf{n} = 0$. Here $V = \int_{\Omega_p} \du V$ is the volume of the physical domain. The last line also shows that $\mathcal{A}$ is not nearly as large as $\Delta t^{-1}$ since the jump $(\mathbf{u}_{\eta}-\mathbf{g})\cdot\mathbf{n}$ is expected to be small.  Finally, we make a remark on the projection of $\mathbf{u}$.  Inside $\Omega_p$, we have 
\begin{eqnarray}
	\nabla \cdot \mathbf{u} &=& \nabla \cdot \mathbf{\tilde{u}} - (\Delta t) \Delta p \\
	&=& (\Delta t) \mathcal{A} \\
	&\leq& C \max_{\mathbf{x} \in \Gamma} |\mathbf{\tilde{u}}_{\eta} - \mathbf{g}|
\end{eqnarray}
where $C$ is an appropriate constant.  Hence any error in the divergence of $\mathbf{u}$ is directly controlled by the error in the velocity boundary condition. In particular, for matching $1$ derivative, we expect $\nabla \cdot \mathbf{u} = O(\eta) + O(\Delta t)$ inside $\Omega_f$. As a result, we can recover a second order scheme, however systematically moving to a higher order method will require an alternative formulation, such as a PPE scheme, for computing the pressure.  

\begin{remark}
	In order to guarantee second order spatial accuracy in algorithm (\ref{Algorithm1}), $\mathbf{\tilde{g}}$ should match at least 1 derivative of $\mathbf{u}_{\eta}$.   
	\myremarkend
\end{remark}

\begin{remark}
	One could also consider solving the Poisson equation (\ref{DiscretePoisson}) with $\mathcal{A} = 0$ and instead impose an interface condition on the normal pressure gradient:
\begin{eqnarray} \label{JumpCondition}
	\lbrack \mathbf{n}\cdot\nabla p_{\eta} \rbrack_{\Gamma} &=&  \frac{1}{\Delta t}\mathbf{n}\cdot(\mathbf{\tilde{u}}_{\eta} - \mathbf{g})
\end{eqnarray}
\begin{eqnarray} \label{JumpDefinition}
	\lbrack f \rbrack_{\Gamma} := \lim_{\epsilon \rightarrow 0} \big( f(\mathbf{x} + \epsilon \mathbf{n}) - f(\mathbf{x} - \epsilon \mathbf{n}) \big)\hspace{4mm} \mathbf{x} \in \Gamma.
\end{eqnarray}
In the definition (\ref{JumpDefinition}), $\mathbf{n}$ is taken as the unit normal directed outward from $\Omega_p$.  Such an approach greatly simplifies the analysis for the behavior of the divergence $(\nabla\cdot \mathbf{u}_{\eta})$ in the resulting PPE scheme.  However, we note that numerically solving (\ref{DiscretePoisson}) with (\ref{JumpCondition}) and $\mathcal{A} = 0$ is harder than simply solving (\ref{DiscretePoisson})--(\ref{SolvabilityCondition}).  Furthermore, the equations also allow for a direct solution using pseudo-spectral methods, while the interface problem does not. \myremarkend
\end{remark}

To test the order of accuracy of the active penalty method, we again use a manufactured solution of the form $\mathbf{u}_e = (u_e, v_e)$ and $p_e$ where
\begin{eqnarray}
	u_e &=& \cos(x)\sin(y)\cos(t) \\
	v_e &=& -\sin(x)\cos(y)\cos(t) \\
	p_e &=& \sin(2x) \cos(y) \cos(t).
\end{eqnarray}

Given initial data corresponding to the exact solution, we numerically evolve the velocity $\mathbf{u}_{\eta}$ and pressure $p_{\eta}$ using the pseudo-spectral method outlined in algorithm \ref{Algorithm1}.  Here we match 1 derivative of $\mathbf{u}_{\eta}$ in the construction of $\mathbf{\tilde{g}}$ and take time steps, with the appropriate restriction, of $\Delta t = O(h^2) = O(\eta)$.  Figure \ref{NavierStokesCVGPlot} shows second order convergence of the velocity field (in $L^{\infty}(\Omega_p)$), as well as the pressure and divergence (in $L^2(\Omega_p)$).  Meanwhile, the pressure and the divergence converge at one order less in $L^{\infty}(\Omega_p)$.  As an example, figures \ref{VelocityError}--\ref{PressureError} show the typical error for velocity and pressure while \ref{FullDivergence}--\ref{DivergenceError} show the velocity divergence.  In addition, \ref{VelocityNS} and \ref{GtildeNS} show the horizontal velocity field along with the horizontal component of the extension $\mathbf{\tilde{g}}\cdot \mathbf{\hat{x}}$.  Note that $\mathbf{u}_{\eta}$ is again very close to $\mathbf{\tilde{g}}$ inside $\Omega_s$.

\begin{figure}[htb!]
	\centering
    \includegraphics[width = \textwidth]{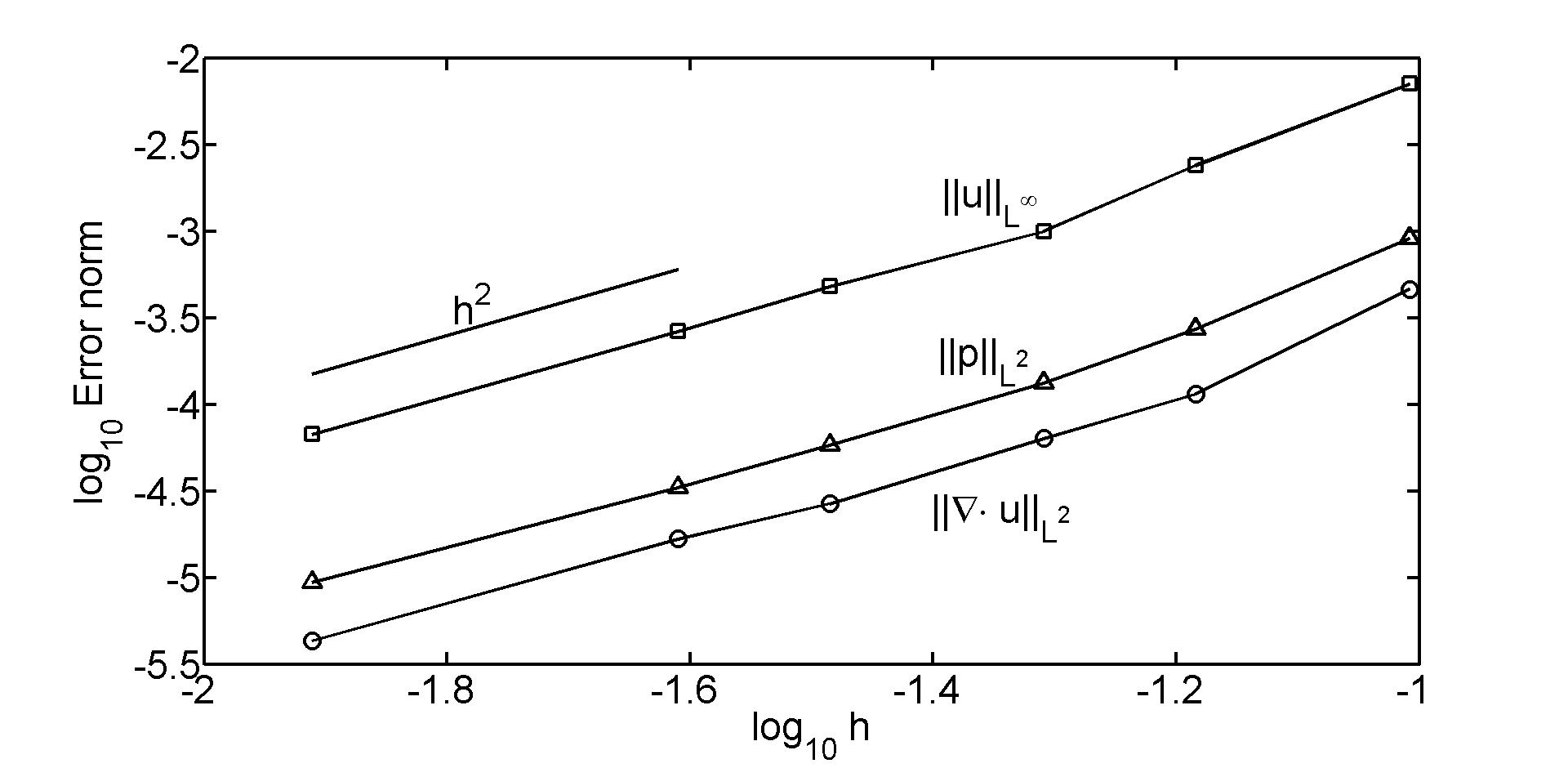} \\
    \caption{Navier-Stokes convergence plot. Second order convergence in $L^{\infty}(\Omega_p)$ for the velocity field (squares), and in $L^2(\Omega_p)$ for the pressure (triangles) and velocity divergence (circles).}
 \label{NavierStokesCVGPlot}
\end{figure}
\begin{figure}[htb!]
	\centering
    \includegraphics[width = \textwidth]{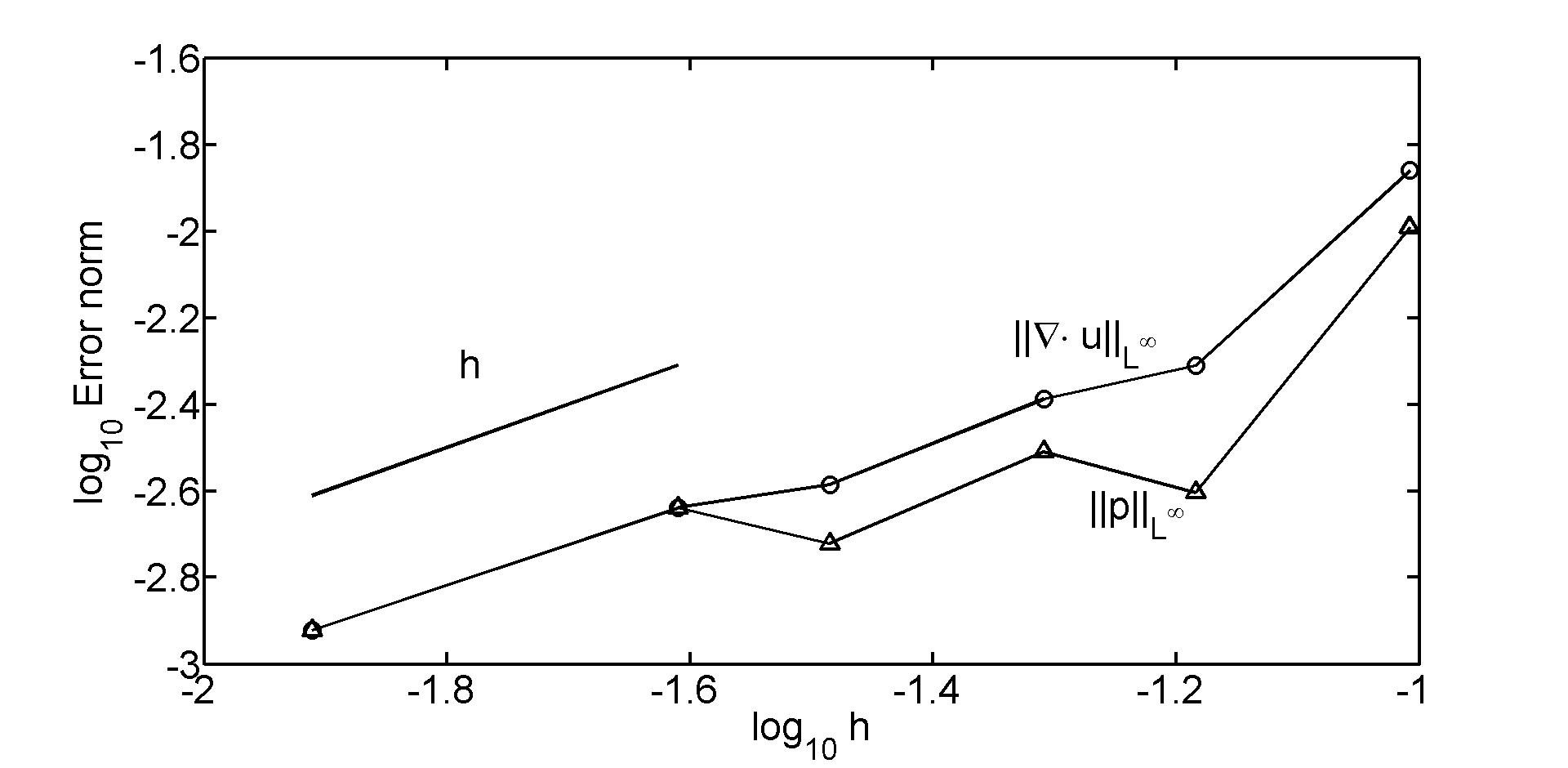} \\
    \caption{Navier-Stokes convergence plot. First order convergence in $L^{\infty}(\Omega_p)$ for the pressure (triangles) and velocity divergence (circles). The weaker convergence in $L^{\infty}(\Omega_p)$ is due to the  boundary layer in the pressure and divergence. The divergence was computed using second order finite differences.}
 \label{NavierStokesCVGPlot_2}
\end{figure}

\begin{figure}
  \centering  \subfloat[Velocity error]{\label{VelocityError}\includegraphics[width=0.5\textwidth]{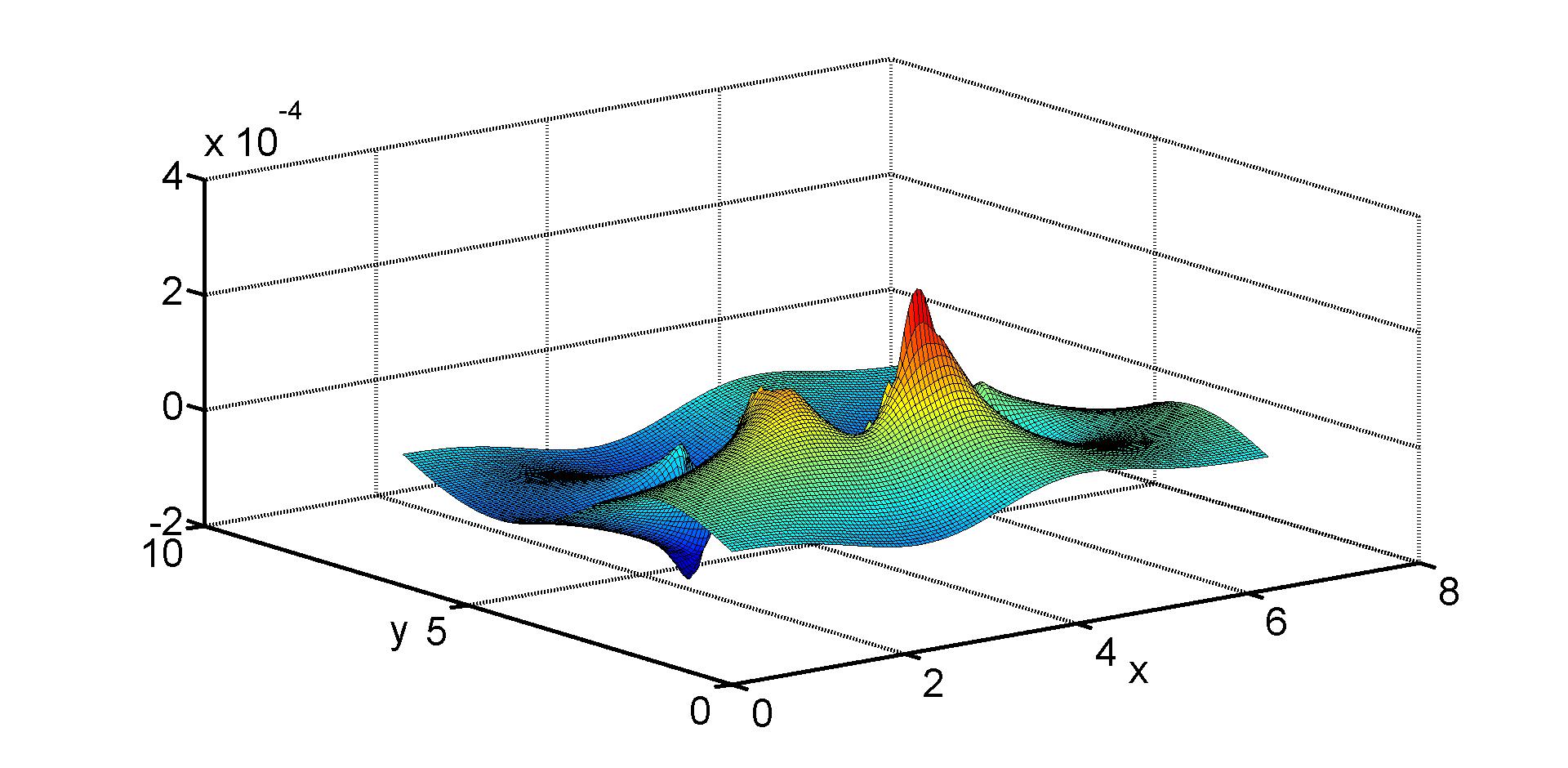}}  \subfloat[Pressure error]{\label{PressureError}\includegraphics[width=0.5\textwidth]{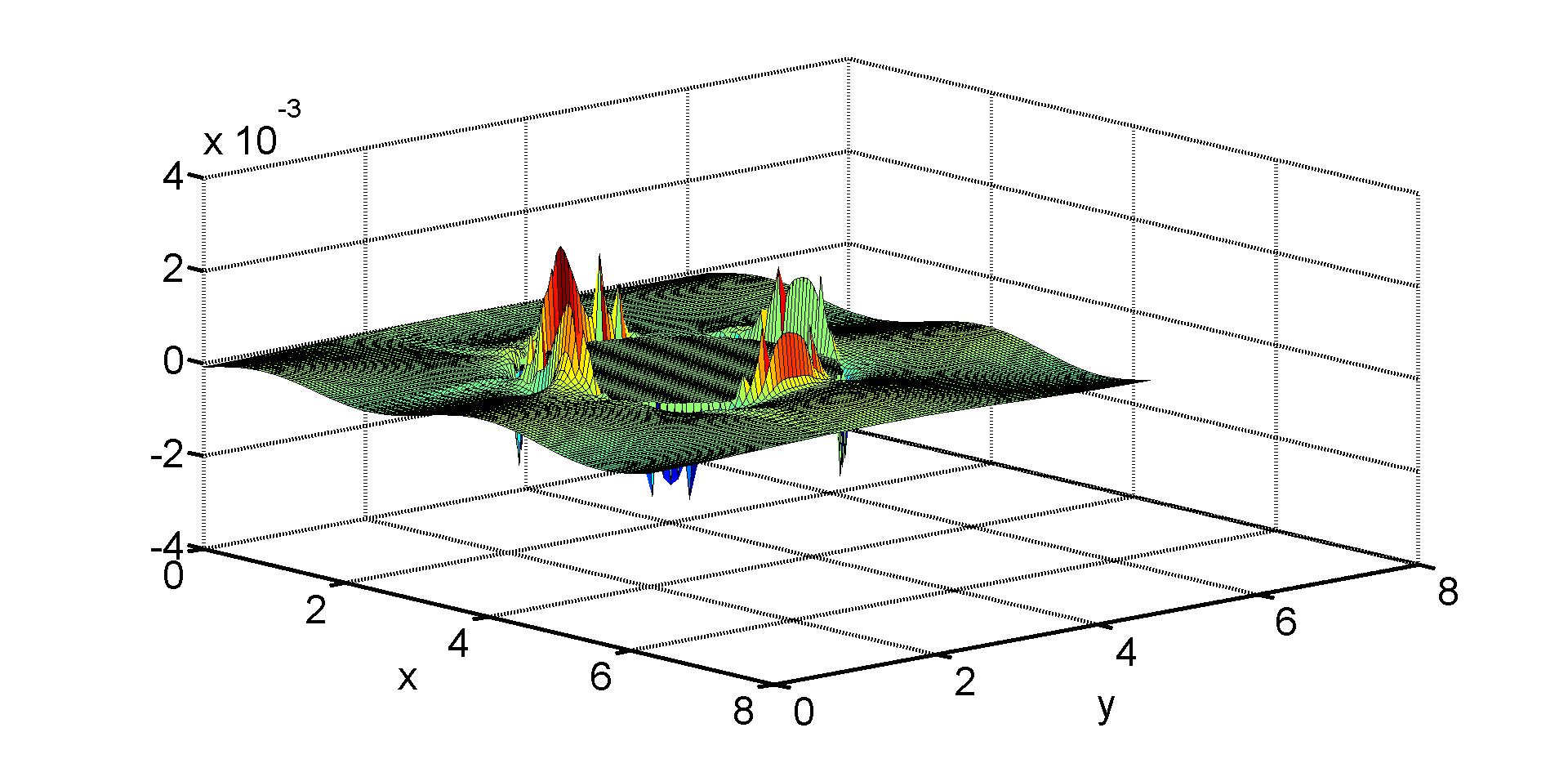}}
  \caption{Error fields with $N = 128$ for the velocity and pressure after $T = 1$.}
\end{figure}

\begin{figure}
  \centering  \subfloat[Full divergence]{\label{FullDivergence}\includegraphics[width=0.5\textwidth]{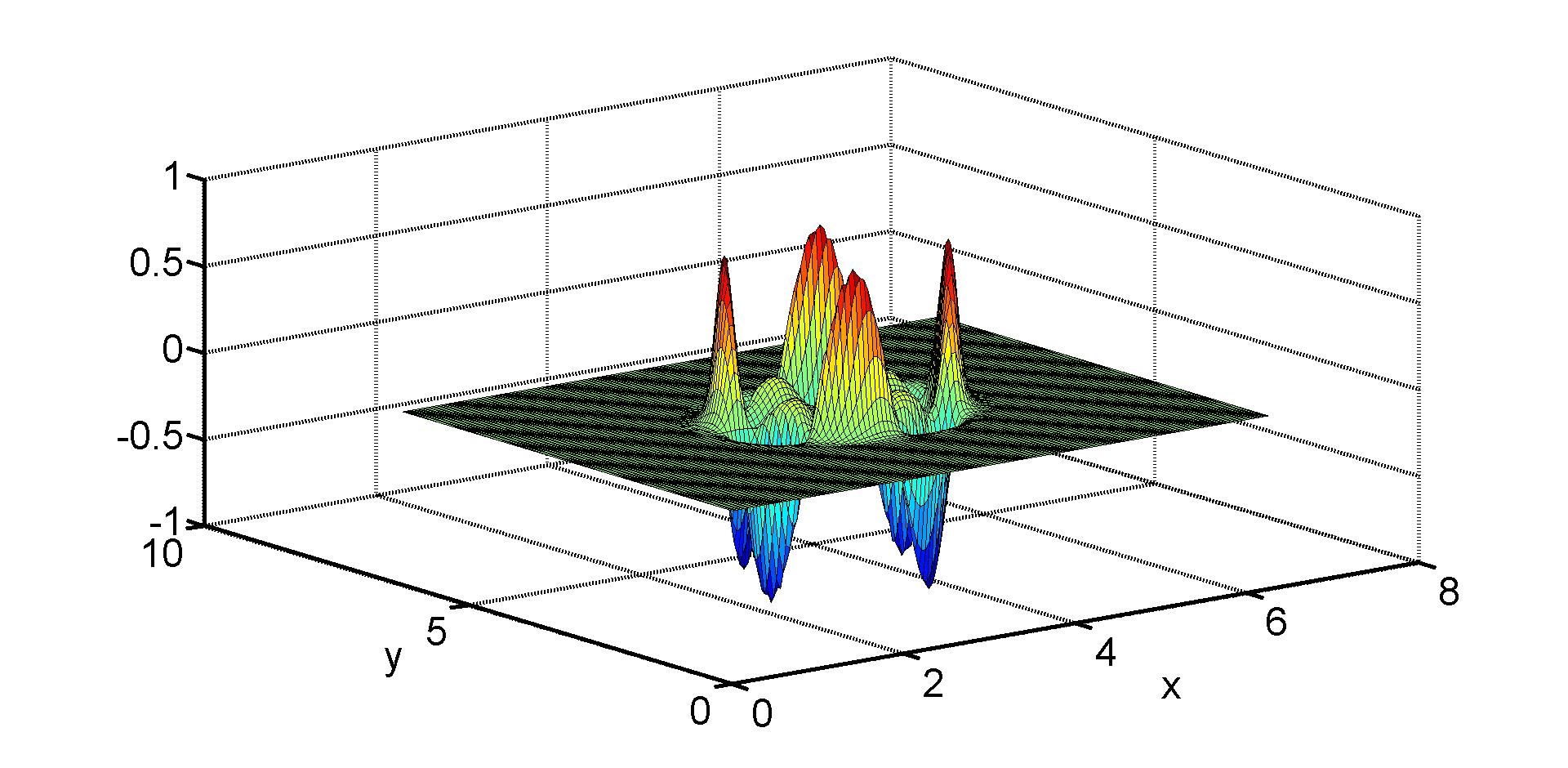}}  \subfloat[Divergence error]{\label{DivergenceError}\includegraphics[width=0.5\textwidth]{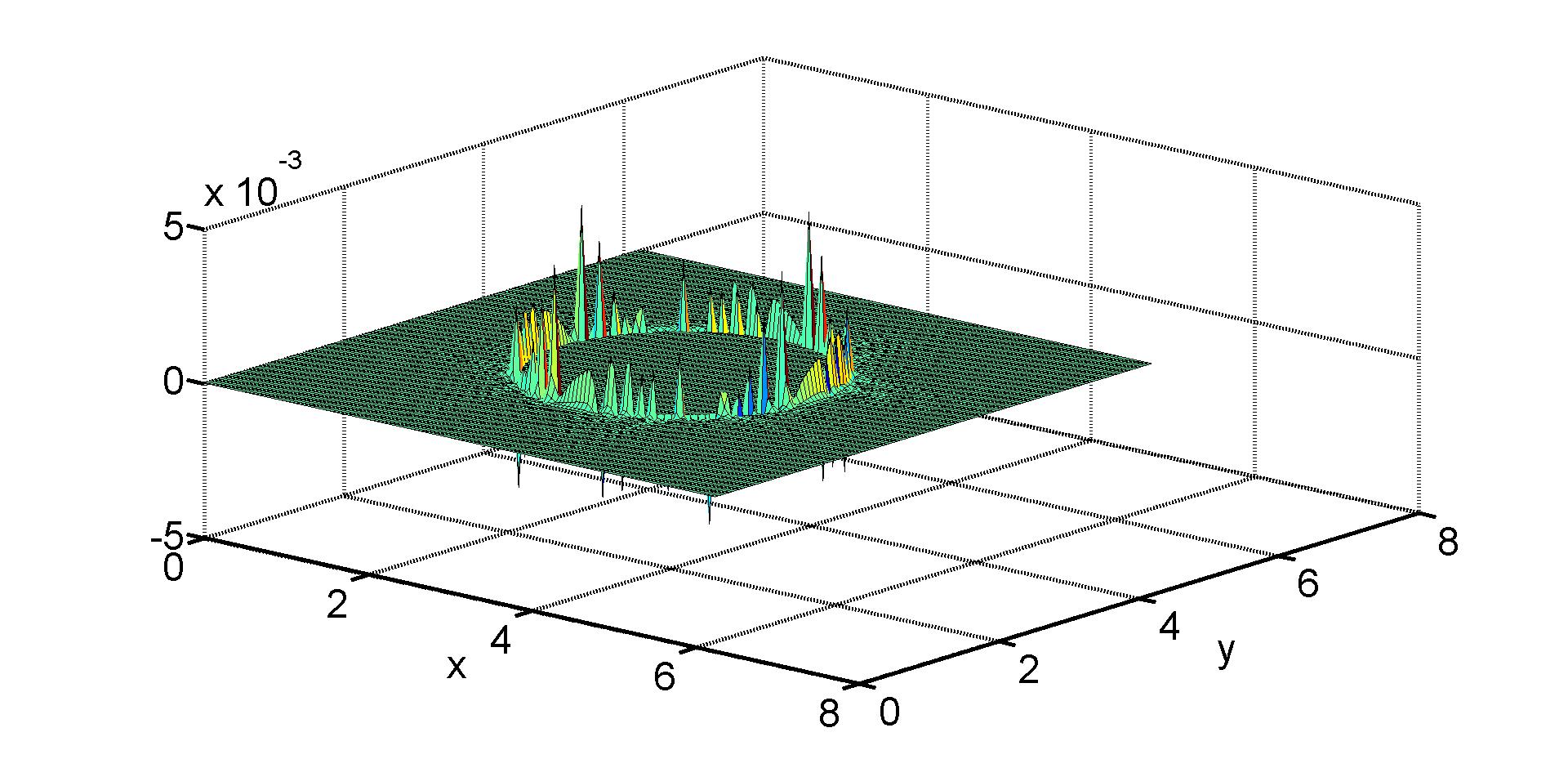}}
  \caption{Plots of the divergence $\nabla \cdot \mathbf{u}_{\eta}$ in $\Omega$ (left) and in $\Omega_p$ (right) with $N = 128$ after $T = 1$. The plot in $\Omega_p$ shows the $||\nabla\cdot\mathbf{u}_{\eta}||_{L^{\infty}(\Omega_p)}$ error occurs at a point in a boundary layer near $\Gamma$.}
\end{figure}

\begin{figure}
  \centering  \subfloat[Velocity $u_{\eta, Num}$]{\label{VelocityNS}\includegraphics[width=0.5\textwidth]{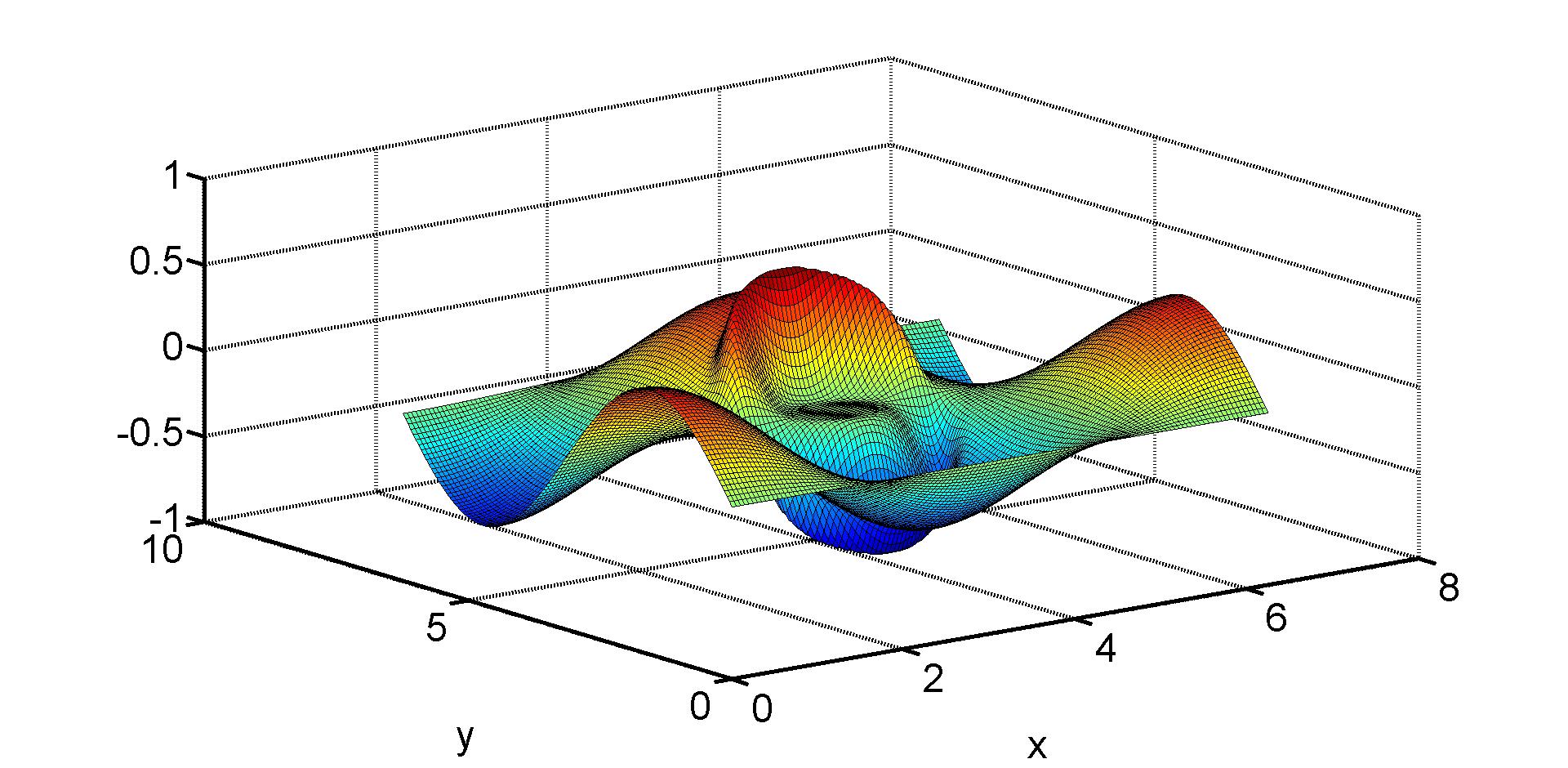}}  \subfloat[The extension $\mathbf{\tilde{g}}\cdot \mathbf{\hat{x}}$]{\label{GtildeNS}\includegraphics[width=0.5\textwidth]{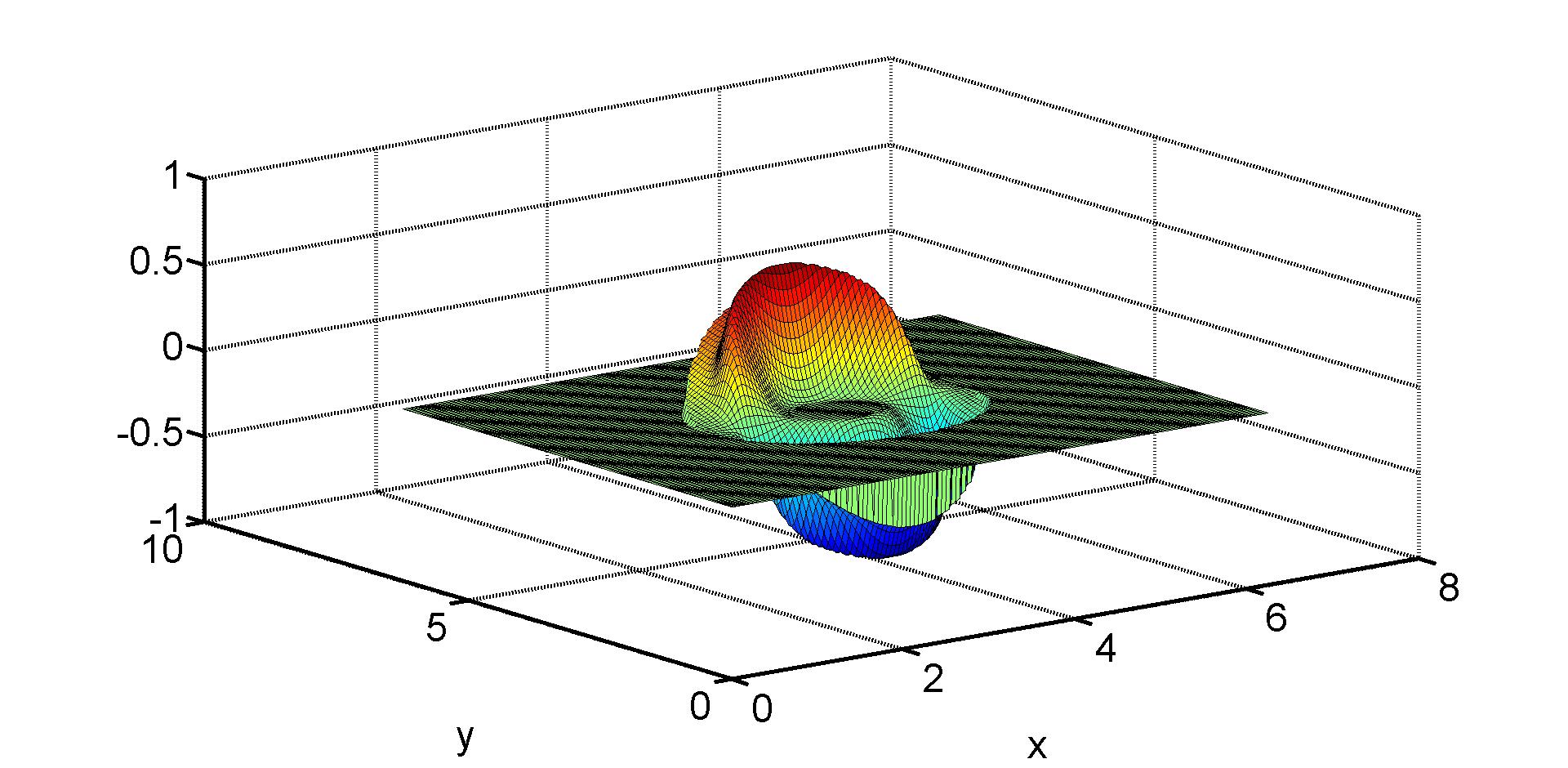}}
  \caption{The numerical velocity field for horizontal component $u_{\eta, Num}$ along with the extension function.  Here $N = 128$ and $T = 1$.}
\end{figure}

\section{Flow around an impulsively started cylinder}

In this section we test our method for the model problem of an impulsively started cylinder \cite{KoumoutsakosLeonard1995}. In this case, we solve the following initial value problem where the fluid starts at rest
\begin{eqnarray}
	\mathbf{u}_{\eta}(\mathbf{x}, 0) = 0 \hspace{5mm} \textrm{for} \, \mathbf{x} \in \Omega. 
\end{eqnarray}
The impulsively started cylinder is then modeled by a moving mask function with a time dependent set $\Omega_s(t)$ and the appropriate Dirichlet boundary condition. For $t > 0$ we have
\begin{eqnarray}
	\Omega_s(t) = \{ \mathbf{x} : |\mathbf{x} - \mathbf{x}_0 - \mathbf{u}_0 t| \leq R \}	\\
	\mathbf{u} = \mathbf{u}_0 \hspace{5mm} \textrm{for} \, \mathbf{x} \in \Gamma. 
\end{eqnarray}
Here $\mathbf{u}_0 = u_0 \mathbf{\hat{e}}_x$ is the velocity of the cylinder, and $(\mathbf{x}_0, R)$ denotes the center and radius of the cylinder.

To simplify the numerical calculation, we perform a Galilean transformation on the coordinates and solve the penalized equations with a stationary mask.  The velocity field then solves the equation
\begin{eqnarray} \label{NS_MovingFrame}
	\partial_t \mathbf{u}_{\eta} + (\mathbf{u}_{\eta} - \mathbf{u}_0)\cdot\nabla \mathbf{u}_{\eta} &=& -\nabla p_{\eta} + \mu \Delta \mathbf{u}_{\eta} + \mathbf{f} - \eta^{-1}\chi_s \; (\mathbf{u}_{\eta} - \mathbf{u}_0 - \tilde{\mathbf{g}}), 
\end{eqnarray}
with initial data $\mathbf{u}(\mathbf{x}, 0) = 0$. Here $\chi_s(\mathbf{x})$ is a stationary mask with $\Omega_s(0)$, while $\mathbf{\tilde{g}}$ is the active penalty term with a zero boundary condition $\mathbf{g} = 0$. 

To compare our results with pre-existing numerical tests, we adopt the following definition of the Reynolds number and time scales from \cite{KoumoutsakosLeonard1995}
\begin{eqnarray}
	\textrm{RE} &=& \frac{2 R u_0}{\mu}  \\
		T &=& \frac{u_0}{R} t.
\end{eqnarray}

Using equation (\ref{NS_MovingFrame}), we then solve for the velocity field in time, and compute the drag force and lift for the impulsively started cylinder.  To compute the force we numerically evaluate the momentum transfer to the fluid 
\begin{eqnarray}
	\mathbf{F}_b &=& -\frac{d}{dt} \int_{\Omega_f} \mathbf{u} \du V \\
				&=& -\frac{d}{dt} \int_{\Omega} \mathbf{u} \big(1 - \chi_s(\mathbf{x})\big) \du V.
\end{eqnarray}
The lift ($C_L$) and drag ($C_D$) coefficients are then evaluated as the non-dimensionalized components of the force
\begin{eqnarray}
	C_D = \frac{\mathbf{F}_b \cdot \mathbf{\hat{e}}_x}{R u_0^2} \qquad 	C_L = \frac{\mathbf{F}_b \cdot \mathbf{\hat{e}}_y}{R u_0^2}.
\end{eqnarray}

In our numerical tests, we examine the impulsively started cylinder for RE $ = 40$ and RE $ = 550$. In both cases, we use $R = 1$, $u_0 = 10$ and the appropriate values of $\mu$ to obtain RE. Here figures \ref{DragRE40} and \ref{DragRE550} show the drag versus time for an impulsively started cylinder with RE $ =40$ and RE $ =550$ respectively.  Note that qualitatively the curves match the benchmark results from \cite{KoumoutsakosLeonard1995}.  In particular, for the RE $ =40$, the drag coefficient monotonically decays to a value slightly below $2$. Meanwhile, for RE $ =550$, the drag first drops, followed by a peak at $T = 3.05$.  Here figure \ref{Snapshots} shows the early development of vorticity for the impulsive cylinder.
We also extend the computation for a much longer time to verify the onset of vortex shedding. Here figure \ref{LiftRE550} shows the oscillations in the lift coefficient versus time, while \ref{vonKarmanStreet} shows the vorticity at various times in the evolution.  We note that due to the periodicity of the domain, the simulation effectively models an array of cylinders, as opposed to the conventional von K\'{a}rm\'{a}n street which arises from flow past one cylinder.

\begin{figure} %
	\includegraphics[width=.9\textwidth]{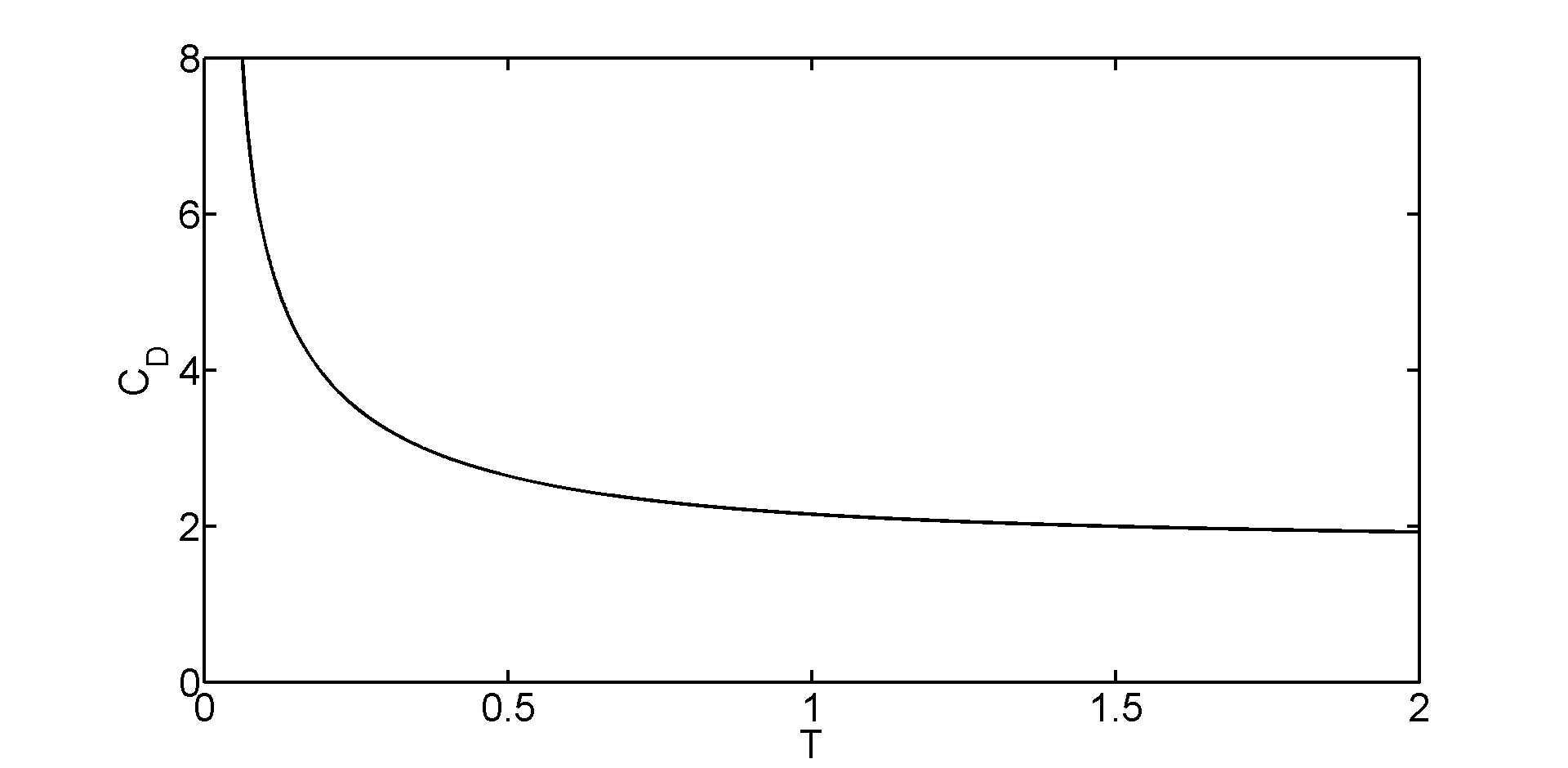} 
\caption{Drag versus time for RE $ = 40$. Here $\eta = 2\times 10^{-4}$, $N = 512$, $l = 0.45$.} \label{DragRE40}
\end{figure}
\begin{figure} %
 \includegraphics[width=.9\textwidth]{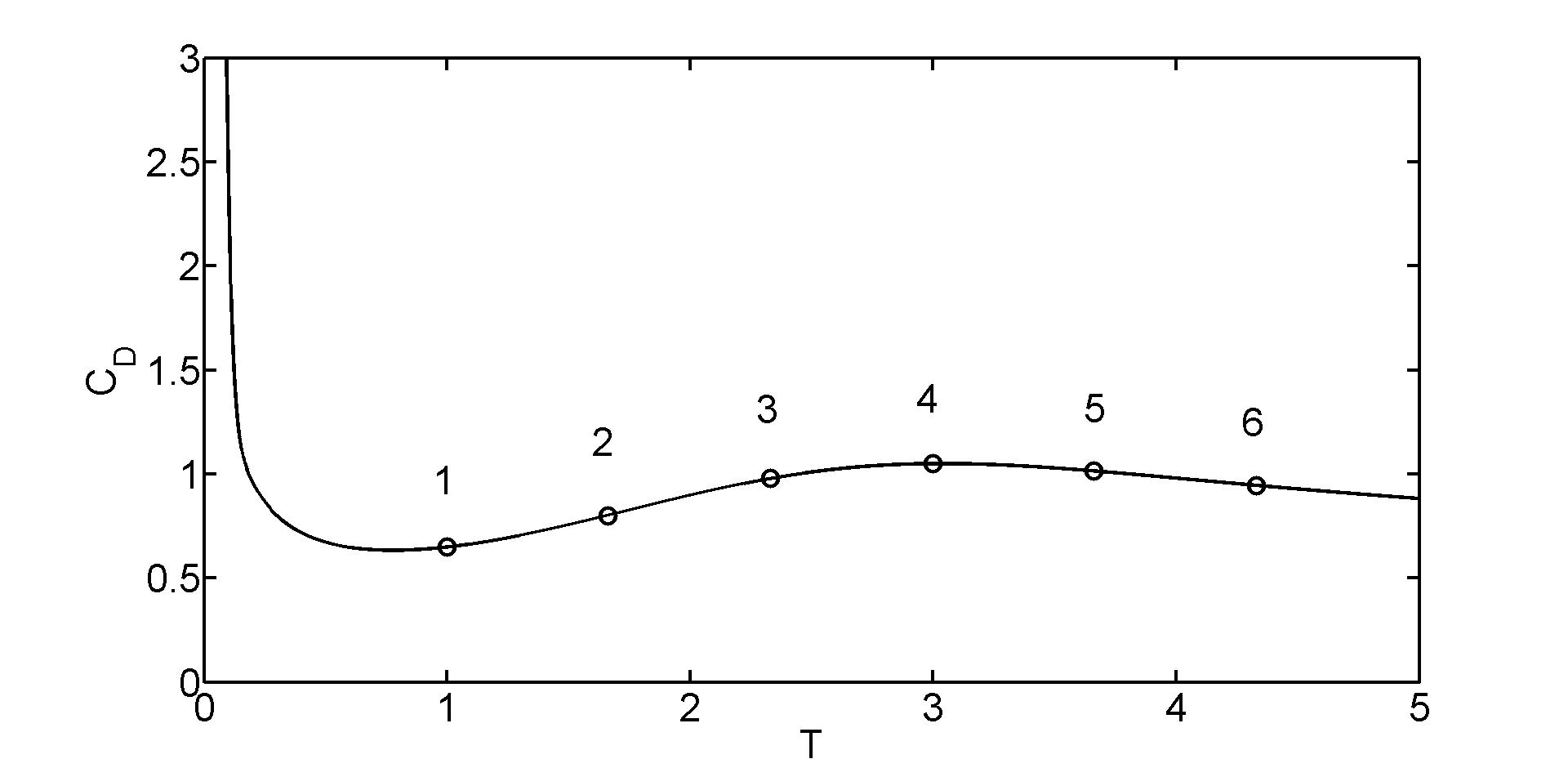} 
\caption{Drag versus time for RE $ = 550$. Here $\eta = 10^{-3}$, $N = 768$, $l = 0.05$. Circles correspond to snapshots of the vorticity shown in figure \ref{Snapshots}.} \label{DragRE550}
\end{figure}

\begin{figure} %
 {{\includegraphics[width=4cm]{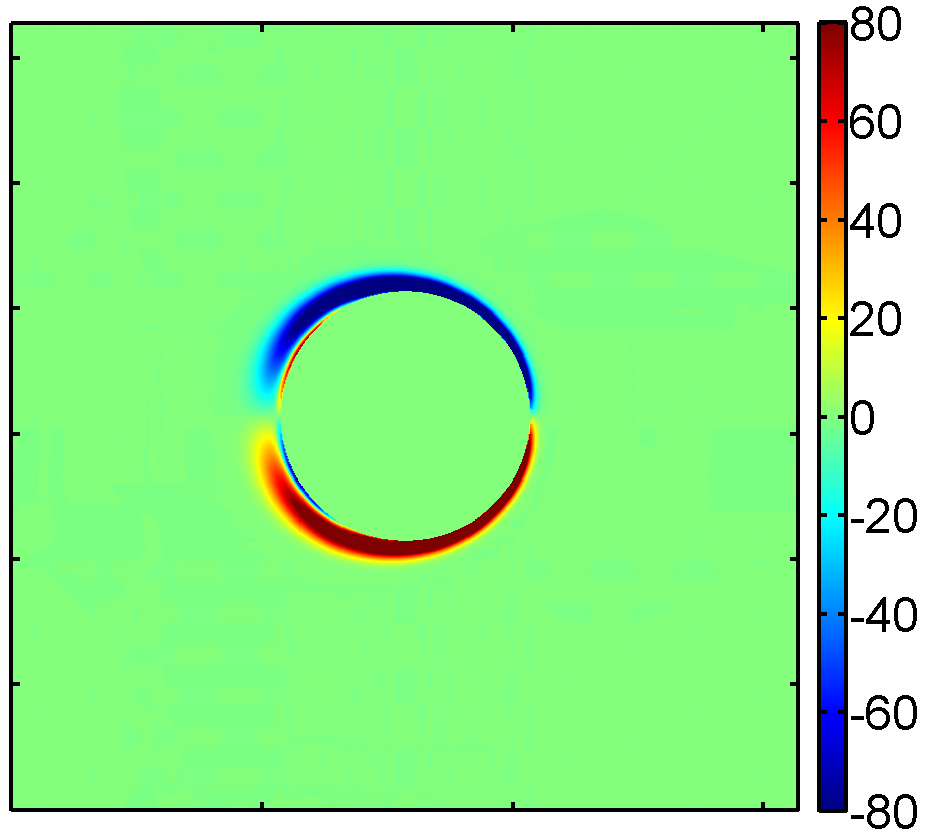}}} 
 {{\includegraphics[width=4cm]{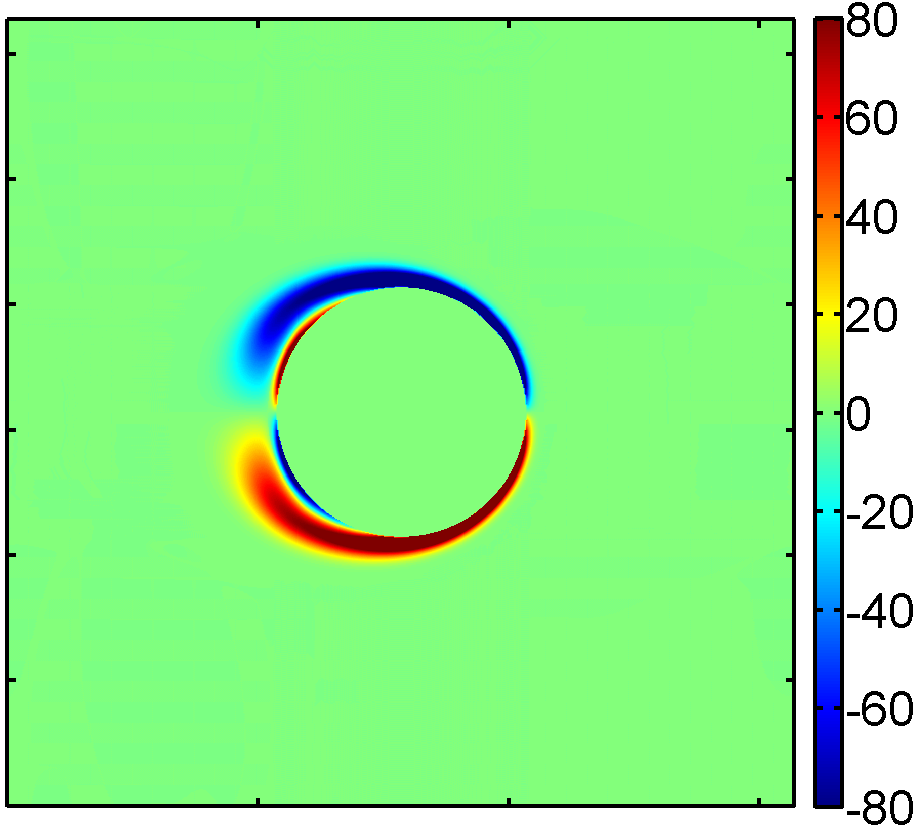}}} 
 {{\includegraphics[width=4cm]{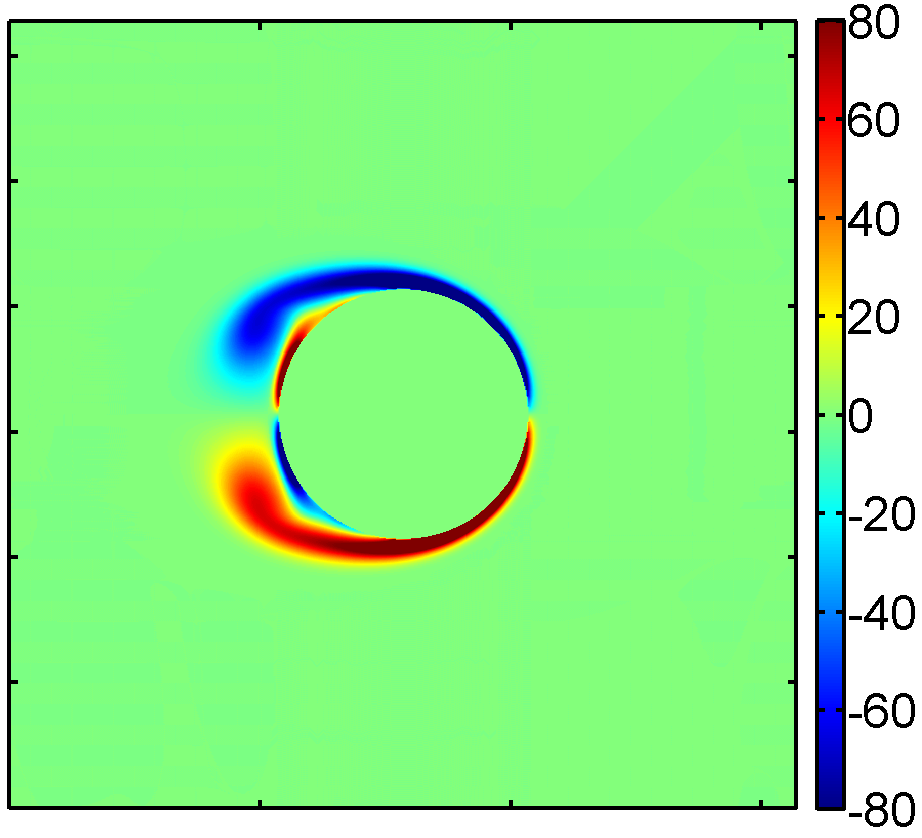}}} \\ 
 {{\includegraphics[width=4cm]{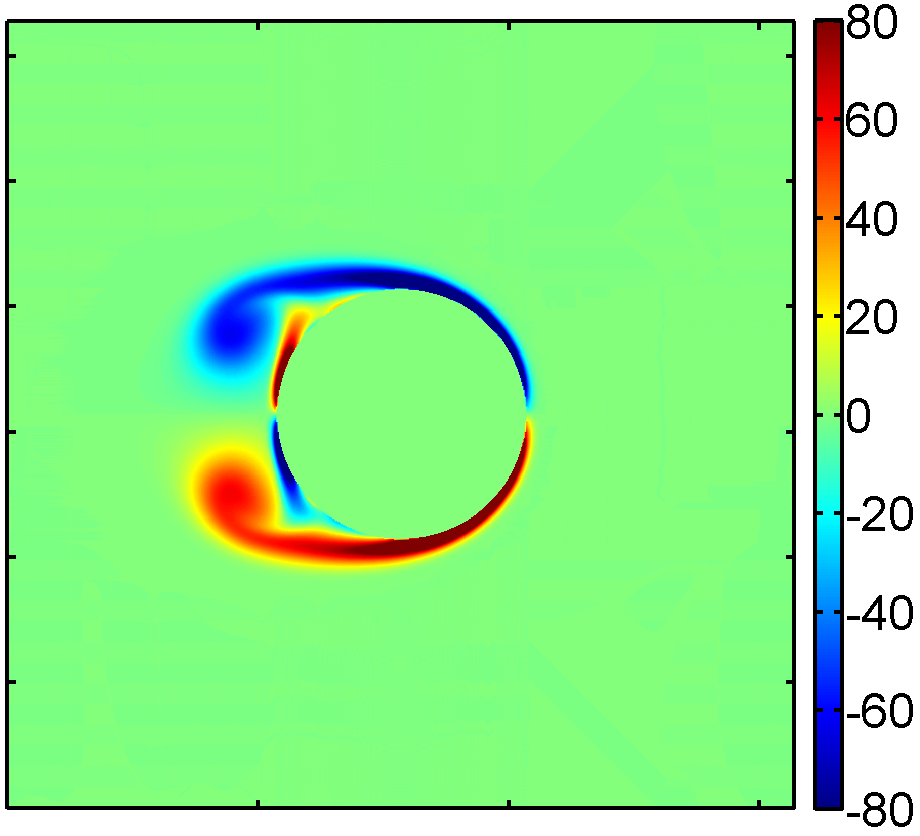}}} 
 {{\includegraphics[width=4cm]{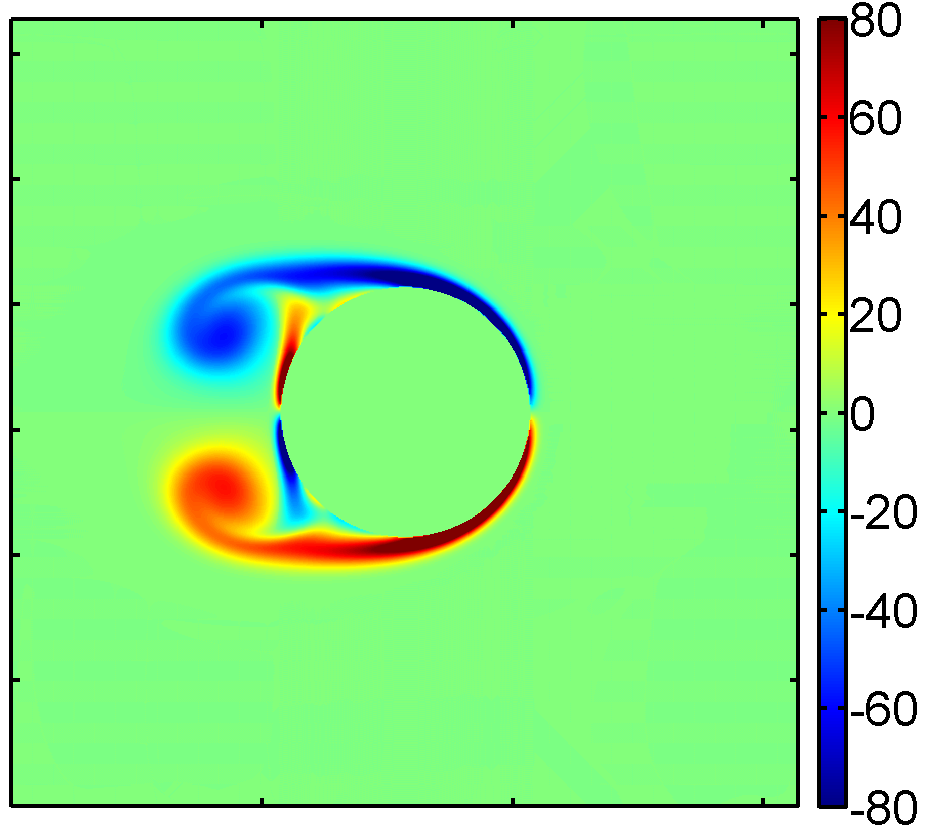}}} 
 {{\includegraphics[width=4cm]{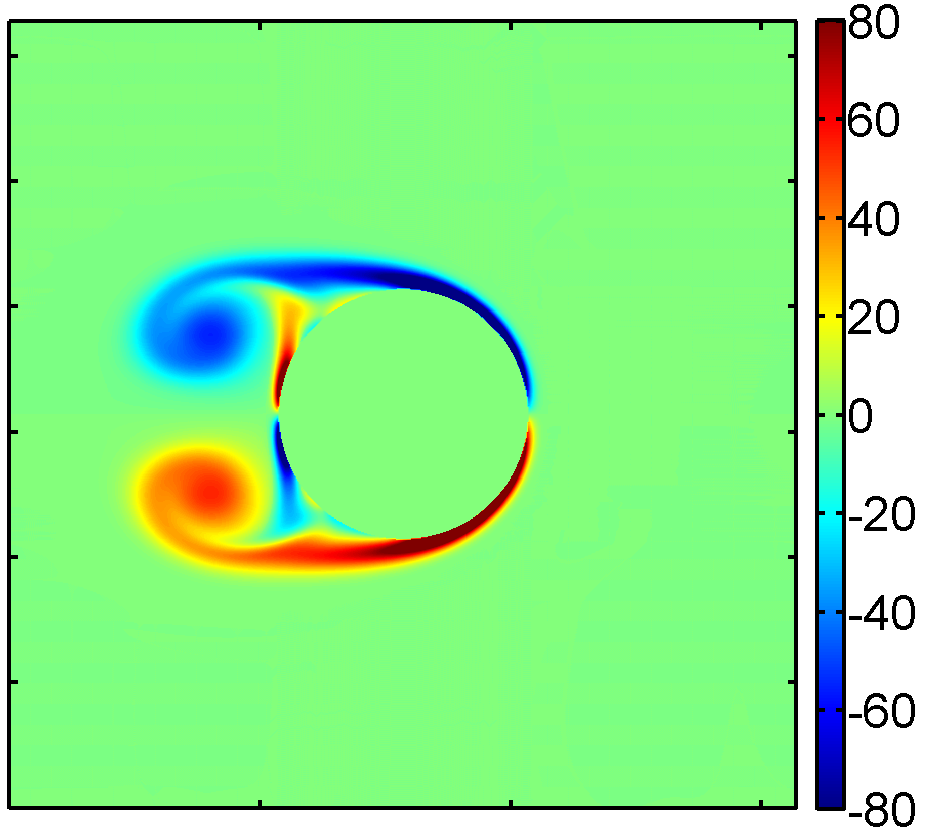}}} 
\caption{Snapshots of the vorticity for an impulsively started cylinder with RE $ = 550$. Images are taken at times (l-r) $T = 1, 1.66, 2.33, 3, 3.66, 4.33$ and correspond to the circles in figure \ref{DragRE550}.}\label{Snapshots}
\end{figure}

\begin{figure} %
 \includegraphics[width=.9\textwidth]{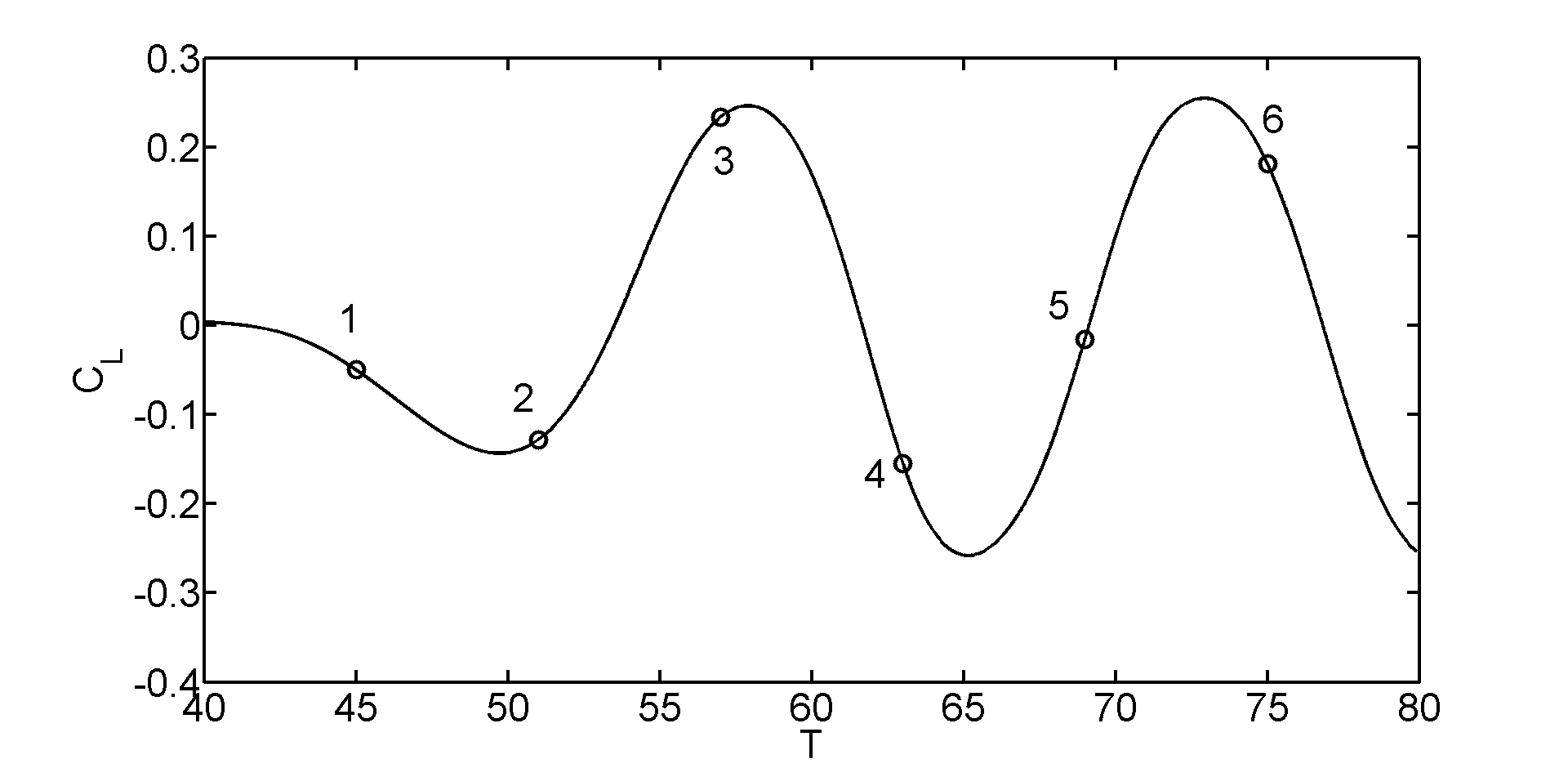} 
\caption{Lift versus time for the onset of the von K\'{a}rm\'{a}m street at RE $= 550$. The oscillations correspond to vortex shedding. Circles correspond to snapshots of the vorticity shown in figure \ref{vonKarmanStreet}.} \label{LiftRE550}
\end{figure}

\begin{figure} %
 {{\includegraphics[width=4cm]{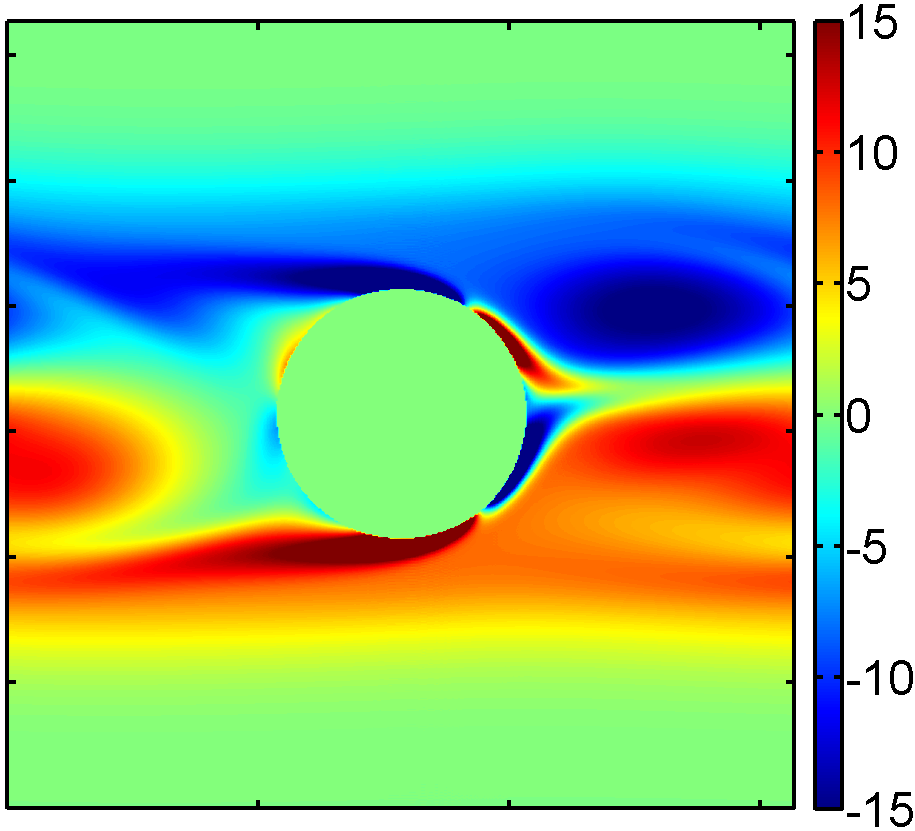}}} 
 {{\includegraphics[width=4cm]{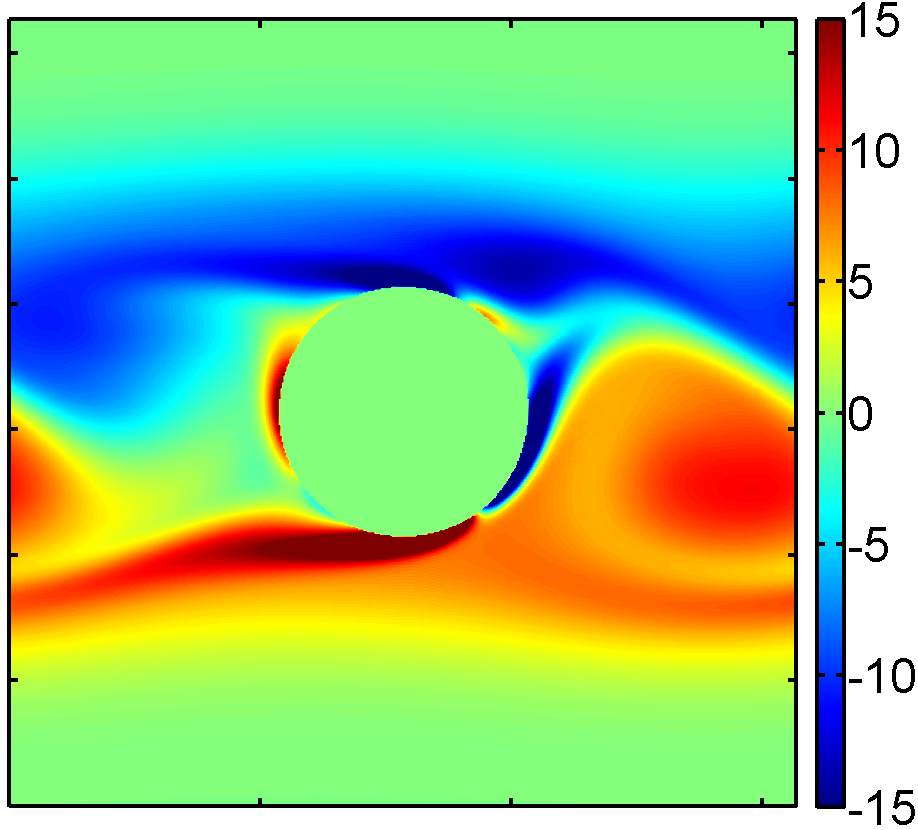}}} 
 {{\includegraphics[width=4cm]{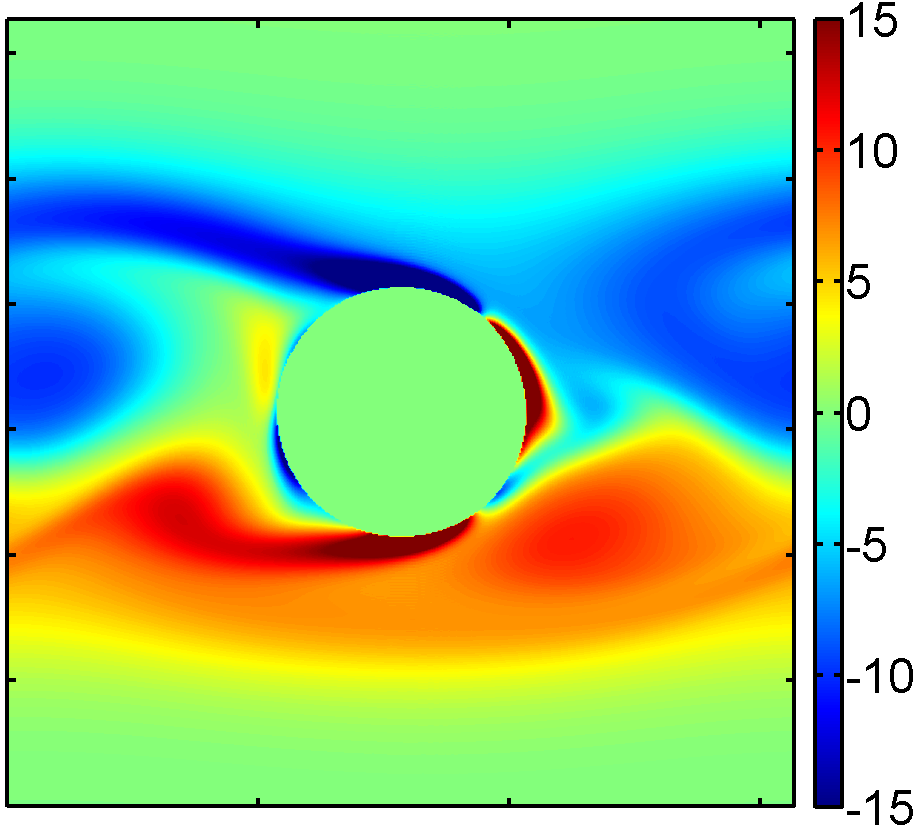}}} \\ 
 {{\includegraphics[width=4cm]{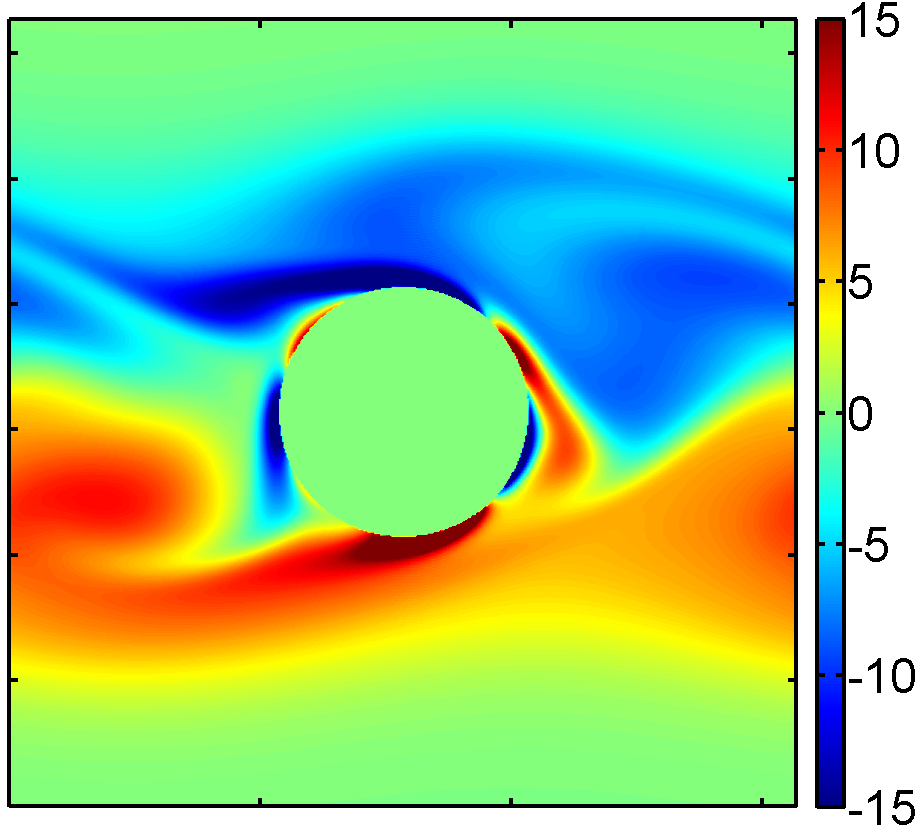}}} 
 {{\includegraphics[width=4cm]{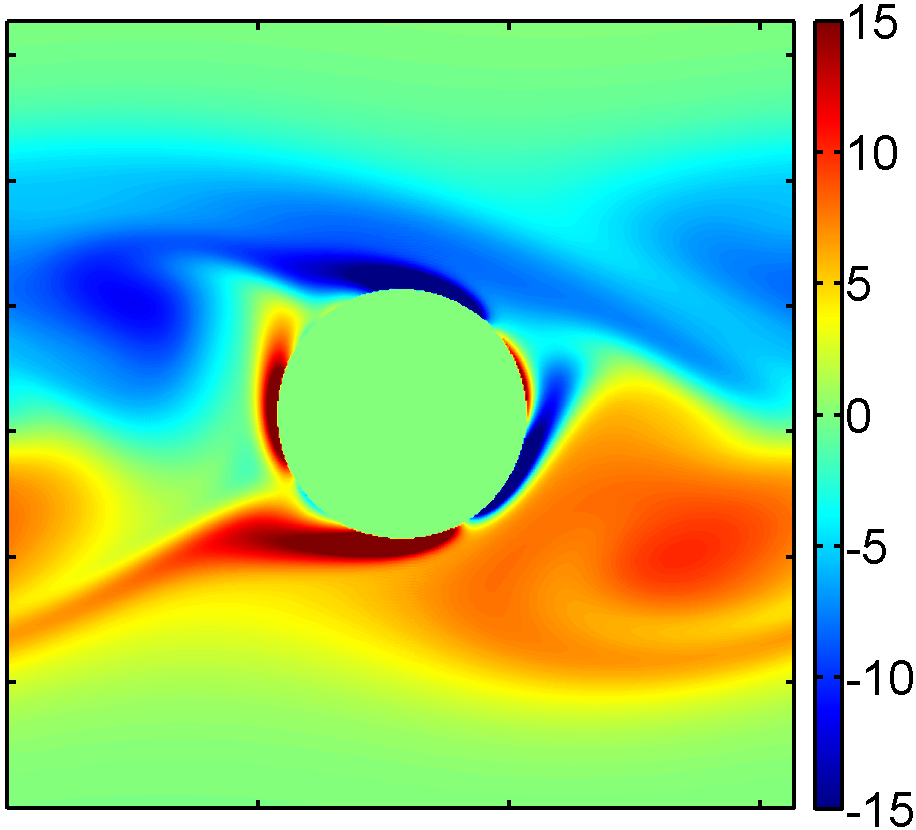}}} 
 {{\includegraphics[width=4cm]{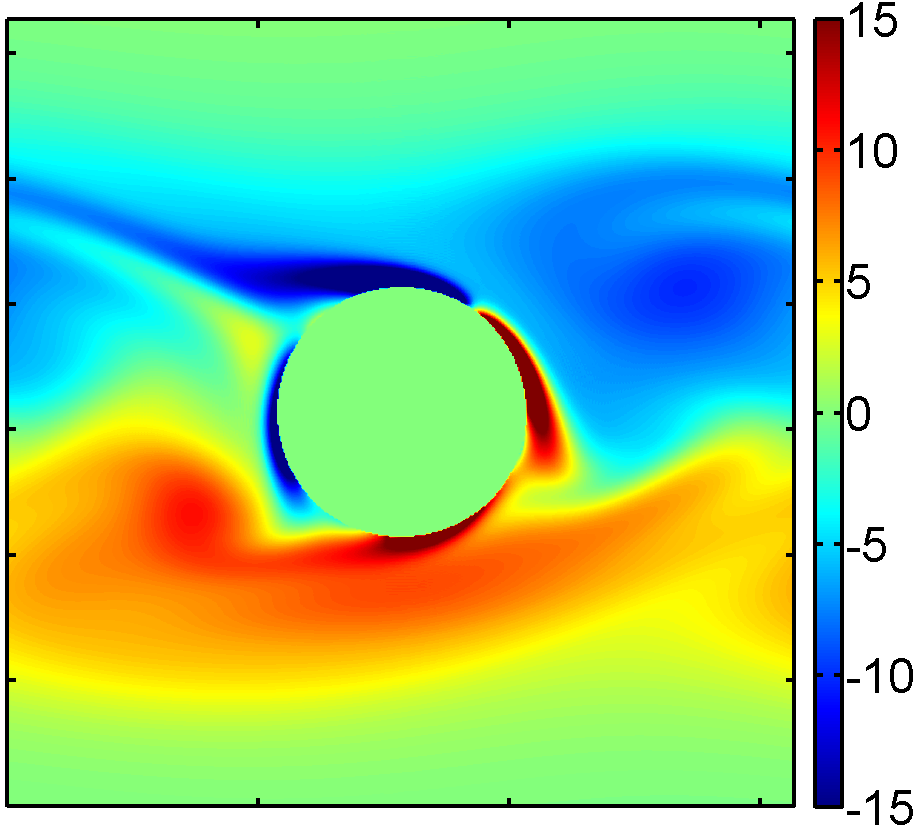}}} 
\caption{Snapshots of the vorticity for an impulsively started cylinder showing the onset of full vortex shedding and von K\'{a}rm\'{a}n type street \cite{vonKarman1963}. Images are for RE $= 550$, and taken at times $T = 45, 51, 57, 63, 69, 75$ corresponding to the circles in figure \ref{LiftRE550}.} \label{vonKarmanStreet}
\end{figure}


\section{Conclusion}

In this paper, we outline how to construct high order penalty methods.  We do so by first introducing an active penalty term for the heat equation.  When we increase the number of matched derivatives, we show that the penalty term improves the analytic convergence rate in terms of the penalty parameter.  Secondly, we examine the numerical stability of the active penalty term.  We show that it does not introduce additional stiffness into the equations or additional length scales that would need to be resolved.  The combination of the high order convergence in the penalty parameter along with the numerical stability then leads to higher order numerical schemes.  Lastly, we extend the penalized term from the heat equation to the incompressible Navier-Stokes equations.  In particular, we show how to handle the divergence constraint on the velocity field. We also conclude with an application of flow around an impulsively started cylinder for RE $=40$ and RE $=550$.  In the case of RE$=550$, we demonstrate the onset of a von K\'{a}rm\'{a}n street.

Although we have outlined a high order approach, there are still remaining issues that limit the practical feasibility of the method.  For instance, at no point do we improve the smoothness of the solution $\mathbf{u}_{\eta}$. In fact the second derivatives of $\mathbf{u}_{\eta}$ remain discontinuous across the curve $\Gamma$, although matching more derivatives in the active penalty term may reduce the size of the discontinuity.  As a result, Fourier methods still have a slow decay in the Fourier modes thereby limiting the ability to spectrally compute derivatives.  In addition, interpolation of high order derivatives in the construction of $\mathbf{\tilde{g}}$ should be one-sided (i.e. from $\Omega_p$) while in practice one would prefer to use points on both sides of $\Gamma$.  As a result, ongoing research includes improving the global smoothness of $\mathbf{u}_{\eta}$ while retaining the high order convergence.

\section{Acknowledgments}
The authors would like to thank Kirill Shmakov and Genevi\`eve Bourgeois for additional preliminary computations not currently presented.  The authors have also greatly benefited from conversations with Dmitry Kolomenskiy, Kai Schneider, Ruben Rosales and Tsogtgerel Gantumur.  This work was supported by an
NSERC Discovery Grant and the NSERC DAS.

\bibliography{main}
\bibliographystyle{plain}
\end{document}